\documentclass[11pt]{article}
\usepackage[a4paper,margin=28mm]{geometry}
\usepackage{amsmath,amssymb,amsthm,mathtools}
\usepackage{enumitem}
\usepackage{booktabs}
\usepackage{placeins}
\usepackage{microtype}
\usepackage{xurl}
\usepackage[hidelinks]{hyperref}

\newtheorem{theorem}{Theorem}[section]
\newtheorem{proposition}[theorem]{Proposition}

\newtheorem{lemma}[theorem]{Lemma}
\newtheorem{remark}[theorem]{Remark}

\newtheorem{example}[theorem]{Example}
\newcommand{\R}{\mathbb R}

\newcommand{\F}{\mathcal F}

\newcommand{\dd}{\,\mathrm d}

\newcommand{\Id}{\mathrm{Id}}
\newcommand{\tr}{\operatorname{tr}}

\title{Computer-assisted local maximality of regular polygons\\
for torsional rigidity}
\author{Beniamin Bogosel\thanks{Faculty of Exact Sciences,
Aurel Vlaicu University of Arad, 2 Elena Dr\u{a}goi Street, Arad, Romania.
Email: \href{mailto:beniamin.bogosel@uav.ro}{\nolinkurl{beniamin.bogosel@uav.ro}}}}
\date{}
\hypersetup{
  pdftitle={Computer-assisted local maximality of regular polygons for torsional rigidity},
  pdfauthor={Beniamin Bogosel},
  pdfsubject={A block-circulant Hessian reduction and validated finite-element proof},
  pdfkeywords={torsional rigidity, regular polygon, shape Hessian, validated numerics, FLINT, finite elements}
}

\begin{document}
\maketitle

\begin{abstract}
We study torsional rigidity as a function of the labeled vertices of a
convex polygon.  Starting from the distributed second shape derivative,
we derive the Hessian with respect to vertex coordinates.  At a regular
polygon, dihedral symmetry makes this matrix block circulant in
radial--tangential coordinates, reducing its spectrum to the eigenvalues
of Hermitian matrices of order two.  We also derive an exact
second-variation Galerkin identity and guaranteed functional residual
majorants.  Finite elements approximate the PDE solutions entering the
Hessian, and FLINT/Arb provides the interval arithmetic needed for
certification.  In the scale-invariant setting, we certify exactly four
zero eigenvalues generated by similarities and $2n-4$ strictly negative
eigenvalues for $5\leq n\leq25$.  The regular polygons in this range are
therefore strict local maximizers, modulo similarities, of torsional
rigidity divided by area squared.  We prove broken regularity of the
first and second material derivatives at each fixed regular polygon,
and deduce an $O(h^2)$ error bound for the exact Galerkin Hessian and
its eigenvalues on meshes fitted to the coarse fan.
\end{abstract}



\section{Introduction}

Let $\Omega\subset\R^2$ be bounded and let $u_\Omega\in H^1_0(\Omega)$
solve $-\Delta u_\Omega=1$.  We use
\[
 J(\Omega)=\frac12\int_\Omega u_\Omega
          =\frac12\int_\Omega|\nabla u_\Omega|^2
\]
throughout; torsional rigidity in the convention $-\Delta u=1$ is $2J$.
Since $J(t\Omega)=t^4J(\Omega)$, the quotient $J(\Omega)/|\Omega|^2$
is invariant under planar similarities.
The Saint--Venant inequality says that the disk uniquely maximizes this
quotient among planar domains \cite{Polya1948}.  P\'olya and Szeg\H{o}
conjectured the fixed-combinatorics analogue: among polygons with a given
number $n$ of sides and given area, torsional rigidity is maximized by the
regular $n$-gon \cite[p.~158]{PolyaSzego1951}.  The cases $n=3,4$ are
classical, whereas the unrestricted global problem remains open for every
$n\geq5$ \cite{GuiHuLi2026}.

Several results address restricted classes of polygons or provide
complementary bounds.  Sharp convex-domain estimates
give the conjecture under an additional quantitative asymmetry condition
\cite{FragalaGazzolaLamboley2013}, and variational symmetry results treat
polygon classes subject to further geometric constraints
\cite{BucurFragala2021}.  Bounds for tangential polygons and polynomial or
Bergman-space approximations provide complementary estimates and numerical
evidence \cite{Keady2021,FleemanSimanek2019,KrausSimanek2024}.  Recent
polytope computations give further nonvalidated evidence
\cite{DiazAvalosLaurain2026}, while flow methods establish monotonicity
results for triangles, rhombi, and rectangles \cite{HuangLiXieYang2025}.
Gui, Hu, Li, and Zhang proved that the torsional rigidity of equal-area
regular polygons increases strictly with $n$ and obtained its large-$n$
asymptotic expansion \cite{GuiHuLiZhang2026}.  Subsequently, Gui, Hu, and Li
proved the fixed-$n$ conjecture, with a quantitative deficit, within the
class of convex tangential polygons \cite{GuiHuLi2026}.  That theorem is
unrestricted for triangles because every triangle is tangential; it does
not settle the full polygon class when $n\geq4$.

Bogosel, Bucur, and Fragal\`a have now proved the local version for every
$n\geq5$: the regular $n$-gon is a strict local maximizer of torsional
rigidity among convex $n$-gons of fixed area
\cite{BogoselBucurFragala2026}.  Their proof is entirely analytic.
The interest of the present paper therefore lies in its numerical
methodology for proving local optimality in discrete geometric
optimization problems.  The combination of shape derivatives,
finite-element error identities, guaranteed residual bounds, and
interval arithmetic provides a reproducible way to certify a geometric
Hessian.  The torsion problem supplies a concrete application, with
explicit spectral enclosures and a proved quadratic Galerkin error rate.

Dahne, G\'omez-Serrano, and Pech-Alberich proved that the first Dirichlet
eigenvalue of equal-area regular polygons decreases strictly with $n$ for
all $n\geq3$, settling the Antunes--Freitas monotonicity conjecture
\cite{DahneGomezSerranoPechAlberich2026}.  Their computer-assisted proof
combines explicit bounds for the large-$n$ asymptotic remainder with
validated eigenvalue enclosures for the remaining finite range.  This
result complements the fixed-$n$ local shape questions considered here.

Our method follows the block-circulant strategy developed by Bogosel and
Bucur for the polygonal Faber--Krahn problem
\cite{BogoselBucur2024,BogoselBucurValidated2024}.  Laurain's distributed
second-order calculus supplies a Hessian with respect to polygonal vertex
parameters \cite{Laurain2020}.  At a regular polygon, dihedral symmetry and
a radial--tangential basis make this Hessian block circulant, and the
discrete Fourier transform reduces its spectrum to $2\times2$ Hermitian
symbols \cite{Tee2007}.  Bogosel and Bucur recorded the analogous torsion
formula in \cite[Remark~7.8]{BogoselBucur2024}; their interval computation in
\cite{BogoselBucurValidated2024} concerns the first Dirichlet eigenvalue,
not torsional rigidity.

For the finite-element discretization errors, the Bogosel--Bucur framework
uses explicit a priori estimates based on interpolation and regularity,
with interval arithmetic controlling the discrete computations.  The
resulting Hessian bounds have a leading $O(h)$ term
\cite[Section~5.2]{BogoselBucurValidated2024}.  This does not mean that the
underlying energy-norm approximation rate is suboptimal: the loss concerns
the estimates for the Hessian.  Here we combine computed equilibrated-flux
residuals with explicit stability bounds, and retain products of errors in
the second variation.  This connects the certificate to guaranteed error
estimation and to the combined a priori/a posteriori approach developed by
Liu and collaborators; see Section~\ref{sec:residual-principle}.  The
$O(h^2)$ Galerkin Hessian rate at each fixed regular polygon is proved in
Section~\ref{sec:rate}; the fixed-mesh certificates do not require an
asymptotic rate.

This paper makes four contributions.  First, we derive explicit formulas
for the three entries of every torsion symbol and isolate the Gram
contribution involving material derivatives.  Second, we express the sector integrals $X,Y,Z$ by
one mixed-torsion scalar and identify its contribution to the symbol.
Third, we derive a residual identity whose error terms are products of
energy errors, and prove the broken regularity needed for $O(h^2)$
Galerkin Hessian convergence at each fixed regular polygon.
Finally, we give a computer-assisted proof of strict local maximality for
$5\leq n\leq25$.  Exact fan reconstruction, continuous residual majorants,
and interval algebraic residuals enclose the continuous Hessian itself, not
only its finite-element discretization.  The cases $n=3,4$ are retained as
independent consistency checks.  All statements are local in the
finite-dimensional space of labeled, simple polygons near the regular
one.  The numerical certificates cover the stated finite range; the
analytic all-$n$ local theorem is due to
\cite{BogoselBucurFragala2026}.

\begin{theorem}[computer-assisted local maximality]\label{thm:n3-n10}
For every $n\in\{5,6,\ldots,25\}$, the Hessian of $J/A^2$ at the regular
$n$-gon, with respect to its $2n$ vertex coordinates, has exactly four zero
eigenvalues generated by similarities and $2n-4$ strictly negative
eigenvalues.  Consequently, the regular $n$-gon is a strict local maximizer
of torsional rigidity among nearby $n$-gons of the same area, modulo
translations and rotations.  Equivalently, it is a strict local maximizer
of $J/A^2$ modulo similarities.
\end{theorem}
The proof and computational tables are given in
Section~\ref{sec:n3-n10}.  The certificates use finite error bounds on the
chosen meshes and do not rely on the convergence rate in
Theorem~\ref{prop:h2}.

\subsection{Disclosure of AI assistance}\label{sec:AIdisclosure}

AI assistants were used under the author's direction for algebraic
calculations, code development, numerical experiments, and drafting.
The author supplied the proof strategy, mathematical notes, references,
and independently developed FreeFEM codes used as numerical benchmarks.
The author reviewed the theoretical arguments and numerical code and
directed the revisions and verification work.  The numerical conclusions
are supported by the explicit error estimates and reproducible interval
checks described in this paper.  Appendix~\ref{app:AIhistory} records the
detailed development history; the prompts and the accompanying three
lectures are supplied in the code repository.

\section{Setting, normalization, and statement}

Let $P\subset\R^2$ be a bounded polygon and let $u=u_P\in H^1_0(P)$ solve
\begin{equation}\label{eq:torsion-state}
 -\Delta u=1\quad\hbox{in }P,\qquad u=0\quad\hbox{on }\partial P.
\end{equation}
We set
\begin{equation}\label{eq:J}
 J(P)=\frac12\int_P |\nabla u|^2
     =\frac12\int_P u.
\end{equation}
Thus torsional rigidity is $2J(P)$, as in the introduction.  If
$t>0$, then
$J(tP)=t^4J(P)$ and $|tP|=t^2|P|$.  The scale-invariant objective is
\begin{equation}\label{eq:F}
                 \F(P)=\frac{J(P)}{|P|^2}.
\end{equation}
The following finite-dimensional criterion gives the local maximality test.

\begin{theorem}[finite-dimensional local criterion]\label{thm:conditional}
Let $P_n$ be the regular $n$-gon.  If the Hessian of \eqref{eq:F}, with
respect to its $2n$ vertex coordinates, has four zero eigenvalues generated
by similarities and $2n-4$ strictly negative eigenvalues, then $P_n$ is a
strict local maximizer among labeled simple $n$-gons in a sufficiently
small vertex neighborhood, modulo translations, rotations, and scalings.
Equivalently, it is a strict local maximizer of $J$ among nearby $n$-gons of
the same area, modulo translations and rotations.
\end{theorem}

\begin{proof}
Use the piecewise-affine fan map from $P_n$ to a nearby labeled polygon.
On a sufficiently small neighborhood its pulled-back coefficient matrix is
uniformly elliptic and depends smoothly on the vertex coordinates.  The
Lax--Milgram solution operator, and hence $J$ and $\F$, are $C^2$ there.
Dihedral symmetry and homogeneity give $D\F(P_n)=0$ as shown below.  The
four-dimensional tangent space to the similarity orbit has a local
complementary slice.  Taylor's formula gives a strict local maximum on that slice
under the stated spectrum condition.  Every sufficiently close orbit
meets the slice, which proves the first assertion.  Intersecting with the
smooth fixed-area hypersurface removes the scale direction and proves the
equivalent constrained assertion.
\end{proof}

The following sections reduce the problem to finitely many $2\times2$
inequalities.  Section~\ref{sec:validated} develops the residual certificate,
Section~\ref{sec:pentagon} presents the pentagon example, and
Section~\ref{sec:n3-n10} reports the certificates for the larger polygons.
These provide an independent computational proof over a finite range,
complementing the analytic theorem of \cite{BogoselBucurFragala2026}.

\section{Vertex calculus and the torsion Hessian}

Write the counterclockwise vertices as $a_i=(x_i,y_i)$, with indices modulo
$n$, and put $z=(x_0,y_0,\ldots,x_{n-1},y_{n-1})$.  The area is
\begin{equation}\label{eq:area}
 A(z)=\frac12\sum_{i=0}^{n-1}(x_i y_{i+1}-y_i x_{i+1}).
\end{equation}
In particular,
\begin{equation}\label{eq:area-gradient}
 \nabla_{a_i}A=\frac12
 \begin{pmatrix}y_{i+1}-y_{i-1}\\x_{i-1}-x_{i+1}\end{pmatrix}.
\end{equation}
The only nonzero $2\times2$ off-diagonal blocks of $D^2A$ join neighboring
vertices.  In scalar coordinates,
\begin{equation}\label{eq:area-hessian}
 A_{x_i y_{i+1}}=A_{y_{i+1}x_i}=\frac12,\qquad
 A_{y_i x_{i+1}}=A_{x_{i+1}y_i}=-\frac12.
\end{equation}

Choose a triangulation of $P$ whose only boundary nodes are the polygon
vertices, and keep its interior nodes fixed.  Let $\phi_i$ be the
continuous, piecewise affine hat function which is one at $a_i$ and zero at
all other nodes.  A vertex displacement
$q=(q_0,\ldots,q_{n-1})$, $q_i\in\R^2$, is extended by
$\theta_q=\sum_i q_i\phi_i$.  The map $\Id+\theta_q$ is affine on every
triangle, fixes the interior nodes, and sends each $a_i$ to $a_i+q_i$;
for small $q$ it is a bi-Lipschitz map onto the displaced polygon.  This
also gives the smooth fixed-domain pullback used in the proof of
Theorem~\ref{thm:conditional}.  Define
$U_i=(U_i^1,U_i^2)\in H^1_0(P;\R^2)$ by
\begin{align}
 \int_P DU_i\nabla v
 ={}&\int_P-(\nabla u\cdot\nabla v)\nabla\phi_i
       +(\nabla\phi_i\cdot\nabla v)\nabla u\notag\\
 &\quad +(\nabla\phi_i\cdot\nabla u)\nabla v+v\nabla\phi_i,
 \qquad v\in H^1_0(P).                                      \label{eq:material}
\end{align}
Then the material derivative of the state in direction $q$ is
$u_q=\sum_iq_i\cdot U_i$.

For $f\equiv1$, Proposition~14 of Laurain \cite{Laurain2020} gives the
distributed second derivative.  We record its planar algebraic reduction.
Put $p=\nabla\phi_i$, $q=\nabla\phi_j$, $g=\nabla u$ and
\[
 S=\left(-\frac{|g|^2}{2}+u\right)I+g\otimes g.
\]
Apart from the material-derivative Gram term, Laurain's local integrand is
\begin{align*}
 &p\otimes(Sq)+(Sp)\otimes q
 +\left(\frac{|g|^2}{2}-u\right)(p\otimes q+q\otimes p)\\
 &\quad -(p\cdot g)q\otimes g-(q\cdot g)g\otimes p
 -(p\cdot q)g\otimes g.
\end{align*}
In two dimensions the identity
\begin{align*}
 &(q\cdot g)p\otimes g+(p\cdot g)g\otimes q
 -(p\cdot g)q\otimes g-(q\cdot g)g\otimes p\\
 &\hspace{35mm}=|g|^2(p\otimes q-q\otimes p)
\end{align*}
reduces this expression to the two local terms below.  Consequently the
$2\times2$ blocks $T_{ij}$ are
\begin{align}
T_{ij}={}&\int_P DU_iDU_j^T
 +\int_P\left(\frac12|\nabla u|^2+u\right)
  (\nabla\phi_i\otimes\nabla\phi_j-
   \nabla\phi_j\otimes\nabla\phi_i) \notag\\
&-\int_P(\nabla\phi_i\cdot\nabla\phi_j)
                  (\nabla u\otimes\nabla u).                 \label{eq:HJ}
\end{align}
Here
\[
 \left(\int_P DU_iDU_j^T\right)_{pq}
 =\int_P\nabla U_i^p\cdot\nabla U_j^q,
 \qquad p,q\in\{1,2\}.
\]
Thus $D^2J(z)=(T_{ij})_{i,j=0}^{n-1}$.  Although individual blocks need not
be symmetric, $T_{ji}=T_{ij}^T$, as required.

\subsection{The scale-invariant and additive Hessians}

Differentiating \eqref{eq:F} twice gives, at an arbitrary polygon,
\begin{align}
D^2\F={}&A^{-2}D^2J
-2A^{-3}(\nabla J\otimes\nabla A+\nabla A\otimes\nabla J)
-2JA^{-3}D^2A
+6JA^{-4}\nabla A\otimes\nabla A.                 \label{eq:HF-general}
\end{align}
At a critical point of $\F$,
\begin{equation}\label{eq:critical-gradient}
             \nabla J=\frac{2J}{A}\nabla A,
\end{equation}
and therefore
\begin{equation}\label{eq:HF-critical}
D^2\F=A^{-2}D^2J-2JA^{-3}D^2A
                    -2JA^{-4}\nabla A\otimes\nabla A.
\end{equation}

For the regular polygon of circumradius one, set
$\vartheta=2\pi/n$ and $a_i=(\cos i\vartheta,\sin i\vartheta)$.  Dihedral
symmetry makes $\nabla_{a_i}J$ radial, and Euler's identity for the
degree-four function $J$ gives $\nabla_{a_i}J=(4J/n)a_i$.  Since
$\nabla_{a_i}A=\sin\vartheta\,a_i$ and
$A=n\sin\vartheta/2$, identity \eqref{eq:critical-gradient} follows.  This
also proves criticality without a numerical PDE calculation.

On the tangent space $\nabla A\cdot q=0$, the sign of $D^2\F$ is the sign of
the additive Lagrangian Hessian
\begin{equation}\label{eq:additive}
 D^2\mathcal L=D^2J-\kappa D^2A,\qquad \kappa=\frac{2J}{A}.
\end{equation}
One must not count its scale direction as a negative constrained direction:
along $tP_n$, $\frac{\dd^2}{\dd t^2}[J(tP_n)-\kappa A(tP_n)]_{t=1}=8J>0$.
The scale-invariant Hessian instead has this direction in its kernel.

\subsection{Dihedral simplification on the regular fan}\label{sec:fan-explicit}

We now specify the coarse finite-element fan.  Let
\begin{equation}
 T_j=\operatorname{conv}\{0,a_j,a_{j+1}\},\qquad
 R_j=\begin{pmatrix}\cos(j\vartheta)&-\sin(j\vartheta)\\
                    \sin(j\vartheta)& \cos(j\vartheta)\end{pmatrix}.
\end{equation}
Let $\phi_j$ be the $P_1$ function on this fan that is one at $a_j$
and zero at the center and at every other polygon vertex.  Thus $\phi_j$ is
supported on $T_{j-1}\cup T_j$.  Up to values on the fan interfaces,
\begin{equation}\label{eq:fan-hat-gradient}
 \nabla\phi_j=\frac{1}{\sin\vartheta}\left\{
 \begin{aligned}
  &(\sin((j+1)\vartheta),-\cos((j+1)\vartheta))&&\text{on }T_j,\\
  &(-\sin((j-1)\vartheta),\cos((j-1)\vartheta))&&\text{on }T_{j-1},\\
  &(0,0)&&\text{elsewhere}.
 \end{aligned}\right.
\end{equation}
This is the piecewise-constant gradient formula used by Bogosel--Bucur.
It shows that the two local terms in \eqref{eq:HJ} vanish unless
$j-i\in\{-1,0,1\}$; the only potentially dense contribution to the Hessian is
$\int DU_iDU_j^T$.

The dihedral covariance is
\begin{equation}\label{eq:dihedral-covariance}
 u(R_jx)=u(x),\qquad
 \phi_j(x)=\phi_0(R_j^Tx),\qquad
 U_j(x)=R_jU_0(R_j^Tx).
\end{equation}
If $S=\operatorname{diag}(1,-1)$ is reflection in the line through $a_0$,
then $U_0(Sx)=S U_0(x)$.  In particular, $U_0^1$ is even and $U_0^2$ is
odd with respect to that reflection.  These identities follow either by
changing variables in \eqref{eq:material}, or by uniqueness of its weak
solution.

All local quantities can be reduced to one sector.  Put
\begin{equation}\label{eq:XYZQ}
 X=\int_{T_0}(\partial_x u)^2,\qquad
 Y=\int_{T_0}(\partial_y u)^2,\qquad
 Z=\int_{T_0}\partial_x u\,\partial_y u.
\end{equation}
Reflection in the bisector of $T_0$, followed by the sectorwise energy
identity, gives
\begin{equation}\label{eq:sector-identities}
 (Y-X)\sin\vartheta+2Z\cos\vartheta=0,\qquad
 X+Y=\int_{T_0}u=\frac{2J}{n}.
\end{equation}
For the energy identity, the radial sides of $T_0$ carry zero normal
derivative by reflection symmetry, while $u=0$ on its polygon side; hence
testing the equation with $u$ on $T_0$ produces no boundary contribution.  In particular,
\[
 \int_{T_0}\left(\frac12|\nabla u|^2+u\right)=\frac{3J}{n}.
\]

\begin{proposition}[one remaining sector scalar]\label{prop:XYZ-anisotropy}
Introduce the unit vectors along and across the bisector of $T_0$,
\begin{equation}
 e_b=(\cos(\vartheta/2),\sin(\vartheta/2)),\qquad
 e_t=(-\sin(\vartheta/2),\cos(\vartheta/2)).
\end{equation}
Set
\begin{equation}\label{eq:sector-anisotropy}
 P_b=\int_{T_0}(\partial_{e_b}u)^2,\qquad
 P_t=\int_{T_0}(\partial_{e_t}u)^2,\qquad
 \Delta_n=P_b-P_t.
\end{equation}
Then
\begin{equation}\label{eq:XYZ-delta}
 X=\frac Jn+\frac{\Delta_n}{2}\cos\vartheta,\qquad
 Y=\frac Jn-\frac{\Delta_n}{2}\cos\vartheta,\qquad
 Z=\frac{\Delta_n}{2}\sin\vartheta.
\end{equation}
Consequently, dihedral symmetry and the energy identity determine
$X,Y,Z$ up to the single anisotropy scalar $\Delta_n$.  In addition,
\begin{equation}\label{eq:delta-elementary-bound}
 |\Delta_n|<\frac{2J}{n},\qquad
 |Z|<\frac{J\sin\vartheta}{n}.
\end{equation}
\end{proposition}
\begin{proof}
Reflection in the bisector diagonalizes
$\int_{T_0}\nabla u\otimes\nabla u$ in the basis $(e_b,e_t)$.  Rotating
$\operatorname{diag}(P_b,P_t)$ back through the angle $\vartheta/2$ and using
$P_b+P_t=X+Y=2J/n$ gives \eqref{eq:XYZ-delta}.  The bounds follow from
$P_b,P_t>0$.
\end{proof}

There is an exact analytic characterization of the remaining scalar, but it
still contains mixed-torsion information.  Put $q=\cot(\vartheta/2)$ and consider
\begin{equation}
 \Omega_q=\{(\xi,\eta):0<\xi<q,\ 0<\eta<1-\xi/q\}.
\end{equation}
Let $u_q$ solve $-\Delta u_q=1$, with Dirichlet condition on $\xi=0$ and
Neumann condition on the other two sides, and set
\begin{equation}
 \tau(q)=\int_{\Omega_q}u_q,\qquad
 E_\xi(q)=\int_{\Omega_q}(\partial_\xi u_q)^2,\qquad
 E_\eta(q)=\int_{\Omega_q}(\partial_\eta u_q)^2.
\end{equation}
Splitting $T_0$ along its bisector gives two congruent right triangles.
After a rigid motion, either half is $\sin(\vartheta/2)\,\Omega_q$: its two legs
have lengths $\sin(\vartheta/2)$ and $\cos(\vartheta/2)=q\sin(\vartheta/2)$.  The polygon side is
the Dirichlet leg, while reflection supplies Neumann conditions on the
bisector and radial side.  Under a dilation by $\lambda$, the state and its
integral scale by $\lambda^2$ and $\lambda^4$, respectively.  Resolving the
two gradient components along the legs therefore gives
\begin{equation}\label{eq:mixed-cell-scaling}
 \frac Jn=\sin^4(\vartheta/2)\,\tau(q),\qquad
 \Delta_n=2\sin^4(\vartheta/2)\,[E_\xi(q)-E_\eta(q)].
\end{equation}
The Hadamard--Pohozaev identity of Gui--Hu--Li--Zhang
\cite[Theorem~1.1]{GuiHuLiZhang2026} states
\begin{equation}\label{eq:mixed-cell-Hadamard}
 \frac{\dd}{\dd q}\left(\frac{\tau(q)}{q^2}\right)
 =\frac{E_\xi(q)-E_\eta(q)}{q^3}>0.
\end{equation}
It follows that
\begin{equation}\label{eq:XYZ-mixed-cell}
\begin{aligned}
 \Delta_n&=2\sin^4(\vartheta/2)\,[q\tau'(q)-2\tau(q)]>0,\\
 X&=\sin^4(\vartheta/2)\{
       \tau(q)+\cos\vartheta[q\tau'(q)-2\tau(q)]\},\\
 Y&=\sin^4(\vartheta/2)\{
       \tau(q)-\cos\vartheta[q\tau'(q)-2\tau(q)]\},\\
 Z&=\sin^4(\vartheta/2)\sin\vartheta[q\tau'(q)-2\tau(q)].
\end{aligned}
\end{equation}
Thus \eqref{eq:XYZ-mixed-cell} is an analytic formula for $X,Y,Z$, and it
also determines the sign of their anisotropy.  In particular, $Z>0$;
$X<Y$ for $n=3$, $X=Y$ for $n=4$, and $X>Y$ for $n\geq5$.  The formula is
analytic but retains the mixed-cell torsion function $\tau$ and its
derivative; no elementary closed form is asserted.

\begin{example}[equilateral triangle]\label{ex:XYZ-triangle}
For $n=3$ and circumradius one, the torsion function is the polynomial
\begin{equation}
 u(x,y)=\frac16\left(x+\frac12\right)
        \big((1-x)^2-3y^2\big).
\end{equation}
Direct integration over
$T_0=\operatorname{conv}\{0,(1,0),(-1/2,\sqrt3/2)\}$ gives
\begin{equation}\label{eq:XYZ-triangle}
 X=\frac{\sqrt3}{384},\qquad
 Y=\frac{13\sqrt3}{1920},\qquad
 Z=\frac1{160},\qquad
 J=\frac{9\sqrt3}{640},\qquad
 \Delta_3=\frac{\sqrt3}{120}.
\end{equation}
For $n=4$, symmetry gives $X=Y=J/4$ and $Z=\Delta_4/2$.
\end{example}

Let
\begin{equation}
 \mathbf J=\begin{pmatrix}0&-1\\1&0\end{pmatrix},\quad
 M_+=\begin{pmatrix}X&Z\\Z&Y\end{pmatrix},\quad
 M_-=\begin{pmatrix}X&-Z\\-Z&Y\end{pmatrix}.
\end{equation}

\begin{lemma}[the three local blocks]\label{lem:local-blocks}
The first block row of the two local terms in \eqref{eq:HJ}, before the
radial--tangential change of basis, is
\begin{align}
 L_0&=-\frac{2}{\sin^2\vartheta}\begin{pmatrix}X&0\\0&Y\end{pmatrix},\notag\\
 L_1&=\frac{\cos\vartheta}{\sin^2\vartheta}M_+-\frac{3J}{n\sin\vartheta}\mathbf J,\notag\\
 L_{-1}&=\frac{\cos\vartheta}{\sin^2\vartheta}M_-+\frac{3J}{n\sin\vartheta}\mathbf J,              \label{eq:local-three-blocks}
\end{align}
and $L_j=0$ otherwise.  Formula \eqref{eq:local-three-blocks} gives an
explicit simplification of \eqref{eq:HJ}: it contains no auxiliary PDE
solution.
\end{lemma}

\begin{proof}
On the only two relevant sectors, \eqref{eq:fan-hat-gradient} reads
\[
 \begin{array}{c|cc}
 &T_0&T_{-1}\\ \hline
 \nabla\phi_0&(1,-\cot\vartheta)&(1,\cot\vartheta)\\
 \nabla\phi_1 &(0,1/\sin\vartheta)&0\\
 \nabla\phi_{-1}&0&(0,-1/\sin\vartheta).
 \end{array}
\]
Reflection in the $x$ axis changes the sector matrix $M_+$ to $M_-$.
The sector energy identity gives the coefficient $3J/n$ in the skew term.
Substitution of the displayed hat gradients into the last two terms of
\eqref{eq:HJ} gives \eqref{eq:local-three-blocks}.  Disjoint supports give
all other zero blocks.
\end{proof}

For completeness, the area terms in \eqref{eq:HF-critical} are just as
explicit.  In radial--tangential coordinates their Fourier symbol is
\begin{equation}\label{eq:area-symbol-explicit}
 \mathcal A_k=\sin\vartheta\cos(k\vartheta)I
              -i \cos\vartheta\sin(k\vartheta)\mathbf J,
\end{equation}
where $\rho_k=\exp(ik\vartheta)$.  Moreover, $\nabla A$ has the constant
radial coordinates $(\sin\vartheta,0)$ at every
vertex.  The Fourier transform of \eqref{eq:local-three-blocks} is
\begin{equation}\label{eq:local-symbol-expanded}
 \sum_{j=-1}^1\rho_k^jL_jR_j
 =\begin{pmatrix}\ell_{x,k}&i\eta_k\\-i\eta_k&\ell_{y,k}\end{pmatrix},
\end{equation}
where
\begin{align}
 \ell_{x,k}&=-\frac{2X}{\sin^2\vartheta}
 +\frac{2\cos\vartheta\cos(k\vartheta)}{\sin^2\vartheta}(X\cos\vartheta+Z\sin\vartheta)+\frac{6J}{n}\cos(k\vartheta),\notag\\
 \ell_{y,k}&=-\frac{2Y}{\sin^2\vartheta}
 +\frac{2\cos\vartheta\cos(k\vartheta)}{\sin^2\vartheta}(Y\cos\vartheta-Z\sin\vartheta)+\frac{6J}{n}\cos(k\vartheta),\label{eq:local-symbol-entries}\\
 \eta_k&=2\cos\vartheta\sin(k\vartheta)\left(\frac{-X\sin\vartheta+Z\cos\vartheta}{\sin^2\vartheta}+\frac{3J}{n\sin\vartheta}\right)
       \notag\\
       &=\frac{4J\cos\vartheta}{n\sin\vartheta}\sin(k\vartheta).\notag
\end{align}
The last equality uses \eqref{eq:sector-identities}; in particular
$-X\sin\vartheta+Z\cos\vartheta=-J\sin\vartheta/n$.  Combining \eqref{eq:local-symbol-entries} and
\eqref{eq:area-symbol-explicit}, and using $A=n\sin\vartheta/2$, produces the exact
cancellation
\begin{equation}\label{eq:explicit-cancellation}
 A^{-2}\sum_{j=-1}^1\rho_k^j L_jR_j
 -2JA^{-3}\mathcal A_k
 =-\frac{2(1-\cos(k\vartheta))}{A^2\sin^2\vartheta}
       \begin{pmatrix}X&0\\0&Y\end{pmatrix},\qquad 0\leq k<n.
\end{equation}
In particular, the explicit off-diagonal term cancels identically.  The
rank-one last term of \eqref{eq:HF-critical} occurs only in mode zero.  On
the area tangent space, multiplying \eqref{eq:explicit-cancellation} by
$A^2$ gives the corresponding formula for the additive Hessian
$D^2J-(2J/A)D^2A$.

\section{Block-circulant Fourier reduction}

With the rotations introduced above, let
\begin{equation}
 \mathsf R=\operatorname{diag}(R_0,\ldots,R_{n-1}).
\end{equation}
The columns of $R_i$ are the radial and tangential directions at $a_i$.
Put
\begin{equation}
                         H=\mathsf R^T(D^2\F)\mathsf R.                       \label{eq:Hrad}
\end{equation}

\begin{proposition}[block-circulant structure]\label{prop:block-circ}
At $P_n$, the matrix $H$ has $2\times2$ blocks
\begin{equation}\label{eq:block-circ}
 H_{ij}=C_{j-i},\qquad C_{-j}=C_j^T,
\end{equation}
where indices are modulo $n$.
\end{proposition}
\begin{proof}
Simultaneously rotating the polygon, the state, and a vertex perturbation by
$\vartheta$ leaves $J$ and $A$ invariant.  In Euclidean vertex coordinates
this conjugates both block indices and both vector components by the same
rotation.  The matrices $R_i$ in \eqref{eq:Hrad} remove that component
rotation, leaving invariance under a cyclic shift of the block indices.
Symmetry of the real Hessian gives the second identity.
\end{proof}

The following Fourier reduction of a block-circulant matrix is due to
Tee~\cite{Tee2007}.  For $k=0,\ldots,n-1$ define the Fourier symbol
\begin{equation}\label{eq:symbol}
 B_k=\sum_{j=0}^{n-1}\rho_k^j C_j.
\end{equation}
By \eqref{eq:block-circ}, $B_k$ is Hermitian.  Reflection symmetry further restricts
its form to
\begin{equation}\label{eq:symbol-entries}
 B_k=\begin{pmatrix}\alpha_k&i\gamma_k\\
                    -i\gamma_k&\beta_k\end{pmatrix},
 \qquad \alpha_k,\beta_k,\gamma_k\in\R.
\end{equation}
Its eigenvalues are
\begin{equation}\label{eq:mode-eigs}
 \mu_k^\pm=\frac12\left(\alpha_k+\beta_k
 \pm\sqrt{(\alpha_k-\beta_k)^2+4\gamma_k^2}\right).
\end{equation}

\begin{proposition}[spectrum and symmetry modes]\label{prop:modes}
The multiset of the $2n$ eigenvalues of $D^2\F(P_n)$ is
$\{\mu_k^-,\mu_k^+:0\leq k<n\}$.  Moreover:
\begin{enumerate}
\item $B_0=0$.  Its radial and tangential vectors generate scaling and
rotation.
\item Each of $B_1$ and $B_{n-1}$ has at least one zero eigenvalue.  The two
corresponding real directions generate the translations.
\end{enumerate}
Consequently, proving strict negativity of every other branch gives
exactly $2n-4$ negative eigenvalues and proves that these four symmetry
directions exhaust the kernel.
\end{proposition}
\begin{proof}
The discrete Fourier transform block-diagonalizes every block-circulant
matrix and gives \eqref{eq:symbol}.  Scaling and rotation have constant
radial--tangential coordinates and hence lie in mode zero.  A constant
Euclidean translation has radial--tangential coordinates proportional to
$(\cos i\vartheta,-\sin i\vartheta)$ or
$(\sin i\vartheta,\cos i\vartheta)$, which are carried by modes $1$ and
$n-1$.  Since $P_n$ is a critical point, differentiating similarity
invariance once more puts all four tangent vectors in the Hessian kernel.
\end{proof}

We can now compute these three entries as explicitly as in the
Bogosel--Bucur reduction.  Write
$a(u,v)=\int_{P_n}\nabla u\cdot\nabla v$ and introduce the radial and
tangential components of the material derivative attached to vertex $j$,
\begin{equation}
 U_{j,r}=\cos(j\vartheta)U_j^1+\sin(j\vartheta)U_j^2,\qquad
 U_{j,t}=-\sin(j\vartheta)U_j^1+\cos(j\vartheta)U_j^2.
\end{equation}
Define the three real numbers
\begin{align}
 G^{rr}_k&=\sum_{j=0}^{n-1}\cos(kj\vartheta)a(U_0^1,U_{j,r}),
 &G^{tt}_k&=\sum_{j=0}^{n-1}\cos(kj\vartheta)a(U_0^2,U_{j,t}),\label{eq:G-three}\\
 G^{rt}_k&=\sum_{j=0}^{n-1}\sin(kj\vartheta)a(U_0^1,U_{j,t}).
\end{align}
The reflection parities following \eqref{eq:dihedral-covariance} remove the
other cosine and sine sums and imply that the lower off-diagonal entry is
$-iG^{rt}_k$.

\begin{proposition}[explicit Fourier entries]\label{prop:explicit-symbol}
For every $k=1,\ldots,n-1$, the entries of the scale-invariant Hessian
symbol $B_k$ in \eqref{eq:symbol-entries} are given by
\begin{align}
 \alpha_k&=\frac1{A^2}\left[G^{rr}_k
       -\frac{2(1-\cos(k\vartheta))}{\sin^2\vartheta}X\right],\label{eq:alpha-explicit}\\
 \beta_k&=\frac1{A^2}\left[G^{tt}_k
       -\frac{2(1-\cos(k\vartheta))}{\sin^2\vartheta}Y\right],\label{eq:beta-explicit}\\
 \gamma_k&=\frac1{A^2}G^{rt}_k.                         \label{eq:gamma-explicit}
\end{align}
Thus the three sums in \eqref{eq:G-three}, all arising from
$\int DU_iDU_j^T$, are the only difficult terms.  Every other contribution
to $\alpha_k,\beta_k,\gamma_k$ is explicit.  For the additive Hessian, the
same formulas hold without the common factor $A^{-2}$.
\end{proposition}
\begin{proof}
The Gram contribution in block $(0,j)$, after rotating its second vector
component, is $a(U_0^p,U_{j,q})_{p,q\in\{r,t\}}$.  Reflection sends $j$ to
$-j$ and gives
\begin{align*}
 a(U_0^1,U_{-j,r})&=a(U_0^1,U_{j,r}),&
 a(U_0^2,U_{-j,t})&=a(U_0^2,U_{j,t}),\\
 a(U_0^1,U_{-j,t})&=-a(U_0^1,U_{j,t}).
\end{align*}
The diagonal sine sums and the off-diagonal cosine sum therefore vanish.
The upper off-diagonal Fourier sum is $iG_k^{rt}$, and Hessian symmetry
gives $-iG_k^{rt}$ below it.  Adding the already computed local and area
contribution \eqref{eq:explicit-cancellation} proves
\eqref{eq:alpha-explicit}--\eqref{eq:gamma-explicit}.
\end{proof}

\begin{remark}
Since $X+Y=2J/n$, Proposition~\ref{prop:explicit-symbol} gives
\begin{equation}\label{eq:trace-no-delta}
 A^2\operatorname{tr}B_k
 =G^{rr}_k+G^{tt}_k-\frac{4(1-\cos(k\vartheta))J}{n\sin^2\vartheta}.
\end{equation}
Thus the explicit sector anisotropy cancels from the trace; it affects only
the difference between the explicit diagonal entries.  The Gram sums still depend on the state.
The value of $Z$ is not needed in the Hessian symbol
$B_k$ in \eqref{eq:symbol-entries}.
\end{remark}

The cancellation of $Z$ can also be seen before taking the final symbol:
\begin{equation}
 -X\sin\vartheta+Z\cos\vartheta=-\frac{J\sin\vartheta}{n},
\end{equation}
which is independent of $\Delta_n$ and is exactly the identity used in
\eqref{eq:local-symbol-entries}.  For $n\geq5$, one has $\cos\vartheta>0$ and
$\Delta_n>0$, so the explicit negative shift in $\alpha_k$ is stronger than
the shift in $\beta_k$; for $n=3$ the ordering is reversed.  When $n=4$,
$\cos\vartheta=0$, both shifts are $-(1-\cos(k\vartheta))J/(2A^2)$, so the explicit
sector anisotropy disappears from the symbol.

The mixed-cell representation gives a version containing only $\tau$ and
its derivative.  Direct substitution gives
\begin{align}\label{eq:mixed-cell-penalties}
 \frac{2(1-\cos(k\vartheta))}{\sin^2\vartheta}X&=\sin^2(k\vartheta/2)\tan^2(\vartheta/2)\,
       [\tau(q)+\cos\vartheta[q\tau'(q)-2\tau(q)]],\\
 \frac{2(1-\cos(k\vartheta))}{\sin^2\vartheta}Y&=\sin^2(k\vartheta/2)\tan^2(\vartheta/2)\,
       [\tau(q)-\cos\vartheta[q\tau'(q)-2\tau(q)]].\notag
\end{align}
Thus the explicit part can be evaluated without integrating three separate
quadratic quantities.

Mode zero requires retaining the rank-one term.  More explicitly,
\begin{equation}\label{eq:mode-zero-gram}
 B_0=A^{-2}\begin{pmatrix}G^{rr}_0&0\\0&G^{tt}_0\end{pmatrix}
 -\frac{2Jn \sin^2\vartheta}{A^4}\begin{pmatrix}1&0\\0&0\end{pmatrix}=0.
\end{equation}
Equivalently, similarity invariance gives the useful exact checks
$G^{rr}_0=2Jn\sin^2\vartheta/A^2=8J/n$ and $G^{tt}_0=0$.

Translation invariance gives additional exact checks in mode one.  Since
the translation vector has complex radial--tangential coordinates $(1,i)$,
\begin{equation}\label{eq:translation-symbol-identities}
 \alpha_1=\beta_1=\gamma_1,\qquad
 B_1=\alpha_1\begin{pmatrix}1&i\\-i&1\end{pmatrix}.
\end{equation}
Consequently its eigenvalues are $\alpha_1+\beta_1$ and zero.  The
$k=n-1$ symbol is its complex conjugate.  Thus a rigorous computation
can insert the translation zero exactly and certify only
$\alpha_1+\beta_1<0$.  If $n$ is even, the
Nyquist mode $k=n/2$ has $G^{rt}_{n/2}=\gamma_{n/2}=0$ and is already
diagonal.

More generally,
\begin{equation}\label{eq:conjugate-symbols}
 \alpha_{n-k}=\alpha_k,\qquad
 \beta_{n-k}=\beta_k,\qquad
 \gamma_{n-k}=-\gamma_k.
\end{equation}
It therefore suffices to compute the modes
$1\leq k\leq\lfloor n/2\rfloor$.
An equivalent test for strict negativity, avoiding square roots, is
\[
 \alpha_k+\beta_k<0,\qquad \alpha_k\beta_k-\gamma_k^2>0.
\]
When evaluating the trace from sector and Gram data, use
\eqref{eq:trace-no-delta} to preserve the exact cancellation.

\section{Validated finite-element computation with FLINT}\label{sec:validated}

We now turn a numerical Hessian into a proof of its sign.  The argument
follows one entry of the Fourier symbol from the PDE solves to an interval
containing its exact value.  Its two analytic ingredients are a residual
estimate for the energy error and a cancellation identity for the Hessian
error.  Outward rounding makes each numerical evaluation of these bounds
an enclosure.

\subsection{What must be certified?}\label{sec:certificate-target}

For each mode, we need bounds for the three entries
$\alpha_k,\beta_k,\gamma_k$ of \eqref{eq:symbol-entries}.
Once they are enclosed, \eqref{eq:mode-eigs} encloses the eigenvalues.
Strict negativity follows when every nonsimilarity branch has an upper
endpoint below zero; Proposition~\ref{prop:modes} supplies the four exact
similarity zeros.

For a running example, take the pentagon and the radial diagonal entry
$\alpha_2$.  Section~\ref{sec:pentagon} reports the enclosure
\[
 \big|\alpha_2-(-0.0147312069261)\big|\leq4.743372\,10^{-4}.
\]
The purpose of this section is to explain how such a radius is proved.
A negative diagonal entry alone does not establish negative definiteness:
the other diagonal and the off-diagonal entry must also be enclosed.

The PDE computation evaluates real Hessian pairings.  The following
proposition identifies the displacements that give each symbol entry.

\begin{proposition}[real Fourier pairings used by the certificate]
\label{prop:real-fourier-pairings}
For $1\leq k<n/2$, define the vertex directions
\begin{align*}
 r^c_{k,j}&=\sqrt{\frac2n}\cos(kj\vartheta)
             (\cos(j\vartheta),\sin(j\vartheta))
             &&\text{(radial direction, cosine)},\\
 t^c_{k,j}&=\sqrt{\frac2n}\cos(kj\vartheta)
             (-\sin(j\vartheta),\cos(j\vartheta))
             &&\text{(tangential direction, cosine)},\\
 t^s_{k,j}&=\sqrt{\frac2n}\sin(kj\vartheta)
             (-\sin(j\vartheta),\cos(j\vartheta))
             &&\text{(tangential direction, sine)}.
\end{align*}
At $P_n$, the entries of \eqref{eq:symbol-entries} satisfy
\begin{equation}\label{eq:real-pairing-symbol}
\begin{aligned}
 D^2\F[r_k^c,r_k^c]&=\alpha_k,&
 D^2\F[t_k^c,t_k^c]&=\beta_k,\\
 D^2\F[r_k^c,t_k^c]&=0,&
 D^2\F[r_k^c,t_k^s]&=\gamma_k.
\end{aligned}
\end{equation}
For even $n$ and $k=n/2$, replace $\sqrt{2/n}\cos(kj\vartheta)$
by $(-1)^j/\sqrt n$ in the cosine directions.  Their pairings still give
$\alpha_k,\beta_k,0$; the sine direction vanishes and $\gamma_k=0$.
\end{proposition}
\begin{proof}
For $k<n/2$, the sums of the squared sines and cosines are $n/2$.
Substituting these normalized directions into $H_{ij}=C_{j-i}$ and using
Fourier orthogonality gives the diagonal entries and, for the two mixed
pairings, $\operatorname{Re}(i\gamma_k)=0$ and
$\operatorname{Im}(i\gamma_k)=\gamma_k$.  For $k=n/2$, the sum of the
squared cosines is $n$ and reflection gives the zero off-diagonal entry.
\end{proof}

\subsection{Why a residual bounds an error}\label{sec:residual-principle}

The finite-element mesh is a uniform subdivision of the coarse fan
$T_j=\operatorname{conv}\{0,a_j,a_{j+1}\}$.  Each coarse edge is divided
into $m$ equal segments, and lines parallel to the edges split each sector
into $m^2$ similar triangles.  Thus the mesh has $nm^2$ triangles, preserves
the polygon's rotations and reflections, and every fine triangle lies
entirely in one coarse sector.  In particular, every fan ray is a union
of mesh edges; this is what \emph{fitted} means below.  We use $h=1/m$
as the refinement parameter, comparable to the maximum element diameter
for fixed $n$.

Begin with the torsion equation itself.  Let $u_h\in V_h$ be its exact
Galerkin solution, where $V_h\subset H_0^1(P_n)$ is the space of continuous
piecewise-affine functions on the fitted mesh.  We use
$a(v,w)=\int_{P_n}\nabla v\cdot\nabla w$ and
$\|v\|_a=\|\nabla v\|_{L^2(P_n)}$; unqualified norms of fields below are
$L^2(P_n)$ norms.

The residual measures the failure of $u_h$ to satisfy the equation for
an arbitrary test function:
\[
 \int_{P_n}v-a(u_h,v)=a(u-u_h,v),\qquad v\in H_0^1(P_n).
\]
It vanishes for $v\in V_h$, but the error $u-u_h$ need not belong to $V_h$.
To bound the residual for every test function, choose an approximation
$y_h$ to the flux $\nabla u$ satisfying the exact balance
$\operatorname{div}y_h=-1$.  We require $y_h\in H(\operatorname{div};P_n)$:
the field and its divergence are square-integrable.  For a piecewise smooth
field this requires matching normal components across mesh edges, so that
integration by parts leaves no interior edge terms.  Hence
\[
 a(u-u_h,v)=\int_{P_n}(y_h-\nabla u_h)\cdot\nabla v.
\]
Taking $v=u-u_h$ and applying Cauchy--Schwarz gives
\begin{equation}\label{eq:equilibrated-majorant}
 \|u-u_h\|_a\leq\|y_h-\nabla u_h\|.
\end{equation}
The unknown solution has disappeared from the right-hand side.  This
guaranteed bound is a consequence of the classical Prager--Synge
hypercircle principle \cite{PragerSynge1947}.  Ern and Vohral\'{i}k
\cite{ErnVohralik2015} develop equilibrated-flux estimates with guaranteed
upper bounds and local efficiency, robust with respect to polynomial
degree.  Their fluxes are obtained from mixed finite-element problems on
vertex patches.  Ern, Smears, and Vohral\'{i}k
\cite{ErnSmearsVohralik2017} treat elliptic problems with $H^{-1}$ source
terms, a setting relevant to the differentiated equations below, whose
right-hand sides act on gradients of test functions.

An admissible flux has the following explicit construction.  For any continuous
piecewise-affine scalar function $\psi_{u,h}$, set
\begin{equation}\label{eq:curl-state-flux}
 y_h=-\frac{x}{2}+\operatorname{curl}\psi_{u,h},\qquad
 \operatorname{curl}\psi=(\partial_2\psi,-\partial_1\psi).
\end{equation}
The first term has divergence $-1$ and the curl has divergence zero.
The normal component of the curl matches across edges because it is the
tangential derivative of a continuous trace.  Choosing $\psi_{u,h}$ by
least squares makes the bound smaller.  Any computed nodal values define
an admissible flux, so the accuracy of that auxiliary solve affects the
sharpness of the bound, not its validity.

The proof of \eqref{eq:equilibrated-majorant} also applies to any conforming
candidate $\widetilde u$ in place of $u_h$.  This observation will allow us
to use rounded coefficient vectors directly.

Guaranteed error estimation also connects a posteriori flux bounds with
computable a priori estimates.  Liu and Oishi \cite{LiuOishi2013}
construct computable a priori bounds for the Poisson finite-element error
on polygonal domains, including nonconvex ones, and use them to enclose
Laplacian eigenvalues.  Li and Liu \cite{LiLiu2018} extend the hypercircle
construction to nonhomogeneous Neumann problems, combining it with an
explicit trace bound.  Liu, Nakao, You, and Oishi
\cite{LiuNakaoYouOishi2021} develop both explicit a posteriori and a priori
estimates for the Stokes equations using an extended hypercircle method.
These works illustrate how a computable flux comparison can also produce
an a priori bound valid for a whole class of right-hand sides.

In the present certificate, the flux mismatch is evaluated for each
computed solution, while explicit bounds for the differentiated forms and
for coercivity control the transfer to the Hessian and the algebraic
errors.  Thus analytical bounds and computed residuals play complementary
roles.  The flux construction above is specific to our two-dimensional
problem.  The additional step is the second-variation identity in
Proposition~\ref{prop:second-residual}, which preserves products of energy
errors when bounding a Hessian entry.

\subsection{From the state equation to a Hessian error}\label{sec:hessian-residual}

Fix two of the vertex displacements in
Proposition~\ref{prop:real-fourier-pairings}; for the running example both
are the radial cosine direction in mode two.  We first describe exact
finite-element solves, and account for their numerical solution in the
next subsection.

Let $q_i,r_i\in\R^2$ be prescribed displacements of each vertex $a_i$.
Extend them continuously over the polygon by affine interpolation on each
fan triangle, taking zero displacement at the center, and denote the
resulting fields by $\theta_q,\theta_r$.  Thus
\[
 \theta_q=\sum_{i=0}^{n-1}\phi_i q_i,\qquad
 \theta_r=\sum_{i=0}^{n-1}\phi_i r_i.
\]
On the fixed regular polygon set
\begin{equation}\label{eq:pullback-map}
 \Phi_{s,t}=\Id+s\theta_q+t\theta_r,\qquad
 D\Phi_{s,t}=I+sD\theta_q+tD\theta_r.
\end{equation}
The derivatives $D\theta_q,D\theta_r$ are constant on each fan triangle.
After composition with $\Phi_{s,t}$, the state on the deformed polygon
satisfies $a_{s,t}(u,v)=\ell_{s,t}(v)$ for every $v\in H_0^1(P_n)$, where
\begin{equation}\label{eq:pulled-forms}
 a_{s,t}(u,v)=\int_{P_n}\mathbb A_{s,t}\nabla u\cdot\nabla v,
 \qquad
 \ell_{s,t}(v)=\int_{P_n}\det(D\Phi_{s,t})v,
\end{equation}
and
\begin{equation}\label{eq:pullback-coefficients}
 \mathbb A_{s,t}
 =\det(D\Phi_{s,t})(D\Phi_{s,t})^{-1}(D\Phi_{s,t})^{-T}.
\end{equation}
At $(s,t)=(0,0)$ these reduce to $a$ and
$\ell(v)=\int_{P_n}v$.

Subscripts $q,r,qr$ on these forms and coefficients denote, respectively,
the derivatives $\partial_s,\partial_t,\partial_s\partial_t$ at
$(s,t)=(0,0)$.  For fixed test functions $v,\zeta$, differentiation gives
\begin{align}
 \ell_q(v)&=\int_{P_n}(\operatorname{div}\theta_q)v,\notag\\
 \ell_{qr}(v)&=\int_{P_n}
 \big[(\operatorname{div}\theta_q)(\operatorname{div}\theta_r)
      -\tr(D\theta_q D\theta_r)\big]v,\label{eq:d-pullback}\\
 a_q(v,\zeta)&=\int_{P_n}\mathbb A_q\nabla v\cdot\nabla\zeta,\notag\\
 a_{qr}(v,\zeta)&=\int_{P_n}\mathbb A_{qr}\nabla v\cdot\nabla\zeta.
 \notag
\end{align}
The first derivatives in direction $r$ follow by replacing $q$ with $r$.
The matrices $\mathbb A_q,\mathbb A_{qr}$ are the corresponding derivatives
of \eqref{eq:pullback-coefficients}; their expansions are recorded in
Appendix~\ref{app:pullback}.  Only their action in these forms and their
norm bounds are needed in the argument.

The Galerkin equation and its derivatives give four linear systems with
the same stiffness matrix.  The first computes the state, the next two its responses
to the chosen vertex displacements, and the fourth its mixed response.
For every $v_h\in V_h$, these equations are

\begin{align}
 a(u_h,v_h)&=\ell(v_h),\notag\\
 a(u_{h,q},v_h)&=\ell_q(v_h)-a_q(u_h,v_h),\notag\\
 a(u_{h,r},v_h)&=\ell_r(v_h)-a_r(u_h,v_h),\label{eq:discrete-differentiated-systems}\\
 a(z_h,v_h)&=\ell_{qr}(v_h)-a_q(u_{h,r},v_h)
 -a_r(u_{h,q},v_h)-a_{qr}(u_h,v_h).\notag
\end{align}
Here $u_{h,q}$ and $u_{h,r}$ are derivatives of the exact discrete
solution.  The exact continuous derivatives are denoted $u_q,u_r$;
in the vertex notation of \eqref{eq:material},
$u_q=\sum_i q_i\cdot U_i$ and similarly for $r$.
When $q=r$, the two first-derivative solves coincide.
The exact discrete energy is $J_h=\frac12\ell(u_h)=\frac12a(u_h,u_h)$;
its mixed derivative is denoted $J_{h,qr}$.

The differentiated forms transfer errors between these equations.  Their
size is controlled by constants satisfying
\begin{equation}
 |a_q(v,\zeta)|\leq C_q\|v\|_a\|\zeta\|_a,\qquad
 |a_r(v,\zeta)|\leq C_r\|v\|_a\|\zeta\|_a,\qquad
 |a_{qr}(v,\zeta)|\leq C_{qr}\|v\|_a\|\zeta\|_a.             \label{eq:Cqr}
\end{equation}
The fourth equation suggests an auxiliary continuous problem with the
exact discrete fields on its right-hand side.  Define
\begin{equation}\label{eq:second-lifting-functional}
 \mathcal R_{qr}^h(v)=\ell_{qr}(v)
 -a_q(u_{h,r},v)-a_r(u_{h,q},v)-a_{qr}(u_h,v),
\end{equation}
and let $Z_{qr}^h\in H_0^1(P_n)$ satisfy
\begin{equation}\label{eq:second-lifting}
 a(Z_{qr}^h,v)=\mathcal R_{qr}^h(v)\quad(v\in H_0^1(P_n)).
\end{equation}
Its Galerkin approximation is precisely $z_h=u_{h,qr}$, the same
$z_h$ as in \eqref{eq:discrete-differentiated-systems}.
The lifting $Z_{qr}^h$ uses discrete data but solves a continuous weak
problem.  Its role is to turn the remaining Hessian error into an energy
pairing.  Its discrete approximation $z_h$ is already available from the
fourth solve; no exact second state derivative needs to be computed.

\begin{proposition}[second-variation residual identity]\label{prop:second-residual}
With the form bounds \eqref{eq:Cqr} and $e=u-u_h$, assume that
\begin{align*}
 \|e\|_a&\leq\delta_0,&
 \|u_q-u_{h,q}\|_a&\leq\delta_q,&
 \|u_r-u_{h,r}\|_a&\leq\delta_r,\\
 \|Z_{qr}^h-z_h\|_a&\leq\eta_{qr},
\end{align*}
then
\begin{equation}\label{eq:second-residual-bound}
 |(J-J_h)_{qr}|
 \leq \delta_q\delta_r+\frac12C_{qr}\delta_0^2
       +\delta_0\eta_{qr}.
\end{equation}
Every quantity on the right can be enclosed by a functional residual
majorant.
\end{proposition}
\begin{proof}
For nearby vertex coordinates $z$, use the same pullback to keep $V_h$
fixed.  Write $a_z$ for the pulled-back form and $e_z=u_z-u_{h,z}$
for the state error.  All state derivatives below are taken after this
pullback, with respect to the deformation parameters in
\eqref{eq:pullback-map}, at $(s,t)=(0,0)$.  In particular,
\[
 u_{qr}=\left.\partial_s\partial_t u\right|_{s=t=0},\qquad
 u_{h,qr}=\left.\partial_s\partial_t u_h\right|_{s=t=0}.
\]
Thus the derivatives of $e=u-u_h$ are
\[
 e_q=u_q-u_{h,q},\qquad e_r=u_r-u_{h,r},\qquad
 e_{qr}=u_{qr}-u_{h,qr}.
\]
Galerkin orthogonality gives the exact energy identity
\begin{equation}\label{eq:energy-error-identity}
       J(z)-J_h(z)=\frac12a_z(e_z,e_z).
\end{equation}
Differentiating twice, including the dependence of the bilinear form on
the polygon, gives the full product-rule expansion
\[
\begin{aligned}
 (J-J_h)_{qr}
 ={}&\frac12a_{qr}(e,e)+a_q(e_r,e)+a_r(e_q,e)\\
    &+a(e_q,e_r)+a(e_{qr},e).
\end{aligned}
\]
Since $V_h$ is fixed, $u_{h,qr}\in V_h$.  Galerkin orthogonality therefore
removes its contribution:
\[
 a(e_{qr},e)=a(u_{qr},e)-a(u_{h,qr},e)=a(u_{qr},e).
\]
To eliminate $u_{qr}$, differentiate the exact state equation twice and
test with $e$:
\[
 a(u_{qr},e)=\ell_{qr}(e)-a_q(u_r,e)-a_r(u_q,e)-a_{qr}(u,e).
\]
Substitute this into the product-rule expansion.  The terms containing
$a_q$ combine to $-a_q(u_{h,r},e)$ because $e_r-u_r=-u_{h,r}$;
the terms containing $a_r$ combine in the same way.  Since $u=u_h+e$,
\[
 \frac12a_{qr}(e,e)-a_{qr}(u,e)
 =-\frac12a_{qr}(e,e)-a_{qr}(u_h,e).
\]
The remaining load and discrete-state terms are exactly
$\mathcal R_{qr}^h(e)$ from \eqref{eq:second-lifting-functional}.  Hence
\begin{equation}\label{eq:second-residual-identity}
 (J-J_h)_{qr}=a(e_q,e_r)-\frac12a_{qr}(e,e)
                  +\mathcal R_{qr}^h(e).
\end{equation}
By \eqref{eq:second-lifting}, Galerkin orthogonality, and
$z_h\in V_h$,
\[
 \mathcal R_{qr}^h(e)=a(Z_{qr}^h,e)
 =a(Z_{qr}^h-z_h,e).
\]
Cauchy--Schwarz now gives \eqref{eq:second-residual-bound}.  In particular,
the unknown error in the second state derivative never has to be bounded
directly.
\end{proof}

\begin{remark}[Role of the second lifting]
The Hessian can be expressed using only the state and its first material
derivatives, as in Bogosel--Bucur~\cite{BogoselBucurValidated2024}.  In the
torsion problem, the variational equation gives
\[
 a(u_q,u_r)=\ell_q(u_r)-a_q(u,u_r).
\]
Thus a second lifting is not required to evaluate the Hessian, nor is it
intrinsic to validation by equilibrated fluxes.  However, the right-hand
side still depends on the unknown state $u$.  Replacing the Gram term by
this expression does not by itself eliminate the state approximation error.

Direct propagation of $O(h)$ energy-error bounds through the individual
Hessian terms generally gives only an $O(h)$ bound; see
Appendix~\ref{app:direct-bounds}.  The second lifting instead preserves
Galerkin cancellation: the remaining state-error term becomes
$a(Z_{qr}^h-z_h,u-u_h)$, and every term in
\eqref{eq:second-residual-bound} is a product of errors.  If
$\delta_0,\delta_q,\delta_r,\eta_{qr}=O(h)$, this gives an $O(h^2)$
Hessian bound.  The gain therefore comes from the residual identity and
the lifting, rather than from equilibrated fluxes alone.  Such fluxes
provide guaranteed energy-error bounds; quadratic decay of the computed
enclosure additionally requires their bounds to have the appropriate
rates and the algebraic errors to be controlled, as discussed in
Section~\ref{sec:rate}.  A certificate using only first material derivatives
remains possible, but may require sharper estimates or finer meshes.
The fixed-mesh validity of \eqref{eq:second-residual-bound} uses no
asymptotic rate assumption.
\end{remark}

\subsection{Account for approximate linear solves}\label{sec:algebraic-errors}

The preceding identities use exact solutions of finite-dimensional
systems.  A floating-point solver supplies candidates.  We distinguish:
\begin{center}
\begin{tabular}{ll}
\toprule
notation & meaning\\
\midrule
$u$ & exact continuous state\\
$u_h$ & exact solution of the finite-element equations\\
$\widetilde u$ & computed finite-element candidate\\
\bottomrule
\end{tabular}
\end{center}
The same distinction applies to the first derivatives.  For the lifting,
$z_h$ is the exact discrete solution and $\widetilde z$ is its candidate;
$Z_{qr}^h$ denotes the solution of the continuous problem with exact
discrete data.
All candidates have exactly zero boundary coefficients, so they belong to
$V_h$.  We must bound both the discretization error $u-u_h$ and the
algebraic error $u_h-\widetilde u$.

Let $Kx_*=f$ be a finite-element system with symmetric positive-definite
stiffness matrix $K$, and let $x$ be a computed coefficient vector.  If a
verified lower bound $\underline\alpha>0$ satisfies
$\lambda_{\min}(K)\geq\underline\alpha$, then
\begin{equation}\label{eq:algebraic-residual}
 \|x-x_*\|_2\leq\frac{\|Kx-f\|_2}{\underline\alpha},
 \qquad
 \|x-x_*\|_K\leq\frac{\|Kx-f\|_2}{\sqrt{\underline\alpha}}.
\end{equation}
Indeed, $K(x-x_*)=Kx-f$ and
$\|x-x_*\|_K^2=(Kx-f)^TK^{-1}(Kx-f)$.  Thus a residual and a lower
eigenvalue bound control the solve error without knowing $x_*$.
For our fan, an explicit $\underline\alpha$ follows from the Poincar\'e
inequality on the square and the finite-element mass matrix; the derivation
is in Appendix~\ref{app:assembly}.

Apply this argument to the four systems in
\eqref{eq:discrete-differentiated-systems}.  Their right-hand sides depend
on earlier solves, so their errors must be transferred in the same order:
state, first derivatives, then lifting.

Define the residual functionals
\begin{align}
 \rho_0&=\ell-a(\widetilde u,\cdot),\notag\\
 \rho_q&=\ell_q-a_q(\widetilde u,\cdot)
                  -a(\widetilde u_q,\cdot),\notag\\
 \rho_r&=\ell_r-a_r(\widetilde u,\cdot)
                  -a(\widetilde u_r,\cdot),\label{eq:four-algebraic-residuals}\\
 \rho_z&=\ell_{qr}-a_q(\widetilde u_r,\cdot)
        -a_r(\widetilde u_q,\cdot)-a_{qr}(\widetilde u,\cdot)
        -a(\widetilde z,\cdot).\notag
\end{align}
Evaluate each residual on the finite-element basis functions.  These values
form the residual vector used in the linear-system error estimate.
Explicitly, if $\varphi_1,\ldots,\varphi_N$ are the interior nodal basis
functions of $V_h$, then
\[
 (r_\rho)_i=\rho(\varphi_i),\qquad
 \rho(v_h)=\sum_{i=1}^N v_i\rho(\varphi_i)
 \quad\text{for }v_h=\sum_{i=1}^N v_i\varphi_i.
\]
For example, the state residual
$\rho_0(v_h)=\ell(v_h)-a(\widetilde u,v_h)$ has vector
\[
 r_{\rho_0}=f-K\widetilde x,
\]
where $\widetilde x$ contains the nodal coefficients of $\widetilde u$.
Thus evaluating a residual functional on $V_h$ is exactly the usual
linear-system residual calculation.

The residual's energy dual norm satisfies
\[
 \|\rho\|_{K^{-1}}=(r_\rho^TK^{-1}r_\rho)^{1/2}
 \leq\frac{\|r_\rho\|_2}{\sqrt{\underline\alpha}}.
\]
This bounds the residual of one linear system; errors in its input fields
must also be transferred.
The form bounds \eqref{eq:Cqr} transfer the state error to the two
first-derivative equations, giving
\begin{equation}\label{eq:algebraic-transfer}
 \epsilon_0=\frac{\|r_{\rho_0}\|_2}{\sqrt{\underline\alpha}},\qquad
 \epsilon_q=\frac{\|r_{\rho_q}\|_2}{\sqrt{\underline\alpha}}+C_q\epsilon_0,\qquad
 \epsilon_r=\frac{\|r_{\rho_r}\|_2}{\sqrt{\underline\alpha}}+C_r\epsilon_0.
\end{equation}

Thus $\epsilon_0,\epsilon_q,\epsilon_r$ bound
$\|u_h-\widetilde u\|_a$, $\|u_{h,q}-\widetilde u_q\|_a$, and
$\|u_{h,r}-\widetilde u_r\|_a$, respectively.

The Hessian center can also be improved using the residuals already
computed.  For the state energy, stationarity makes the corrected energy
error quadratic.  The next proposition gives the corresponding Hessian
correction: it cancels the terms linear in the first-derivative solve
errors.  The lifting residual is paired with the state error.

\begin{proposition}[residual-corrected discrete centers]
\label{prop:corrected-centers}
Define the corrected Hessian and energy centers
\begin{align*}
 \widetilde C_{\rm corr}={}&\ell_{qr}(\widetilde u)+\ell_q(\widetilde u_r)
 -\frac12a_{qr}(\widetilde u,\widetilde u)
 -a_q(\widetilde u,\widetilde u_r)\\
 &\quad+\rho_r(\widetilde u_q)+\rho_0(\widetilde z),\\
 \widetilde J_{\rm corr}={}&\ell(\widetilde u)
       -\frac12a(\widetilde u,\widetilde u).
\end{align*}
Then the exact discrete quantities satisfy
\begin{align}
 J_{h,qr}-\widetilde C_{\rm corr}
  & =\rho_z(u_h-\widetilde u)
      -\frac12a_{qr}(u_h-\widetilde u,u_h-\widetilde u)\notag\\
  &\quad+a(u_{h,q}-\widetilde u_q,u_{h,r}-\widetilde u_r),
                                                        \label{eq:center-identity}\\
 0\leq J_h-\widetilde J_{\rm corr}
  & =\frac12\|u_h-\widetilde u\|_a^2.                 \label{eq:J-center-identity}
\end{align}
Consequently
\begin{equation}\label{eq:center-algebraic-bound}
\begin{aligned}
 |J_{h,qr}-\widetilde C_{\rm corr}|
 &\leq E_{\rm alg}:=\frac{\|r_{\rho_z}\|_2}{\sqrt{\underline\alpha}}\epsilon_0
       +\frac12C_{qr}\epsilon_0^2+\epsilon_q\epsilon_r,
 \\
 |J_h-\widetilde J_{\rm corr}|&\leq\frac12\epsilon_0^2.
\end{aligned}
\end{equation}
\end{proposition}
\begin{proof}
The residual definitions and \eqref{eq:discrete-differentiated-systems}
give, for every $v_h$,
\begin{align*}
 a(u_h-\widetilde u,v_h)&=\rho_0(v_h),\\
 a(u_{h,q}-\widetilde u_q,v_h)&=\rho_q(v_h)-a_q(u_h-\widetilde u,v_h),\\
 a(u_{h,r}-\widetilde u_r,v_h)&=\rho_r(v_h)-a_r(u_h-\widetilde u,v_h),\\
 a(z_h-\widetilde z,v_h)&=\rho_z(v_h)-a_q(u_{h,r}-\widetilde u_r,v_h)\\
 &\quad-a_r(u_{h,q}-\widetilde u_q,v_h)-a_{qr}(u_h-\widetilde u,v_h).
\end{align*}
Expand the differentiated discrete energy about the four candidates, use
these identities with $v_h=u_h-\widetilde u$, $u_{h,q}-\widetilde u_q$,
and $u_{h,r}-\widetilde u_r$, and use symmetry of the differentiated forms.
All terms linear in the first-derivative solve errors cancel, leaving
\eqref{eq:center-identity}.
Equation~\eqref{eq:J-center-identity} is the usual stationary energy
identity.  Cauchy--Schwarz, \eqref{eq:Cqr}, and
\eqref{eq:algebraic-transfer} give \eqref{eq:center-algebraic-bound}.
\end{proof}

\subsection{Assemble the final enclosure}\label{sec:entry-enclosure}

The remaining task is to supply the four energy bounds in
Proposition~\ref{prop:second-residual} from the computed candidates.
Let $\Phi_0,\Phi_q,\Phi_r,\Phi_z$ denote their flux mismatches, defined
explicitly in Appendix~\ref{app:fluxes}.  Each is an $L^2$ norm measuring
how far a candidate gradient, with the differentiated source terms, is
from an exactly equilibrated flux.  We use these norms to bound the state
and first-derivative discretization errors and the lifting error in
Proposition~\ref{prop:second-residual}.  Integration by parts, as in
Section~\ref{sec:residual-principle}, and the algebraic bounds above give
\begin{align}
 \widehat\delta_0&=\Phi_0,&
 \widehat\delta_q&=\Phi_q+C_q\Phi_0+\epsilon_q,&
 \widehat\delta_r&=\Phi_r+C_r\Phi_0+\epsilon_r,                    \label{eq:combined-deltas}\\
 \widehat\eta_{qr}&=\Phi_z+C_q\epsilon_r+C_r\epsilon_q
                          +C_{qr}\epsilon_0.\notag
\end{align}
The state mismatch $\Phi_0$ bounds $\|u-\widetilde u\|_a$, so Galerkin best
approximation gives $\|u-u_h\|_a\leq\Phi_0$.  The quantity $\Phi_q+C_q\Phi_0$ bounds
$\|u_q-\widetilde u_q\|_a$: its second term transfers the flux residual from
the candidate state $\widetilde u$ to the exact state $u$.  Separately,
$\epsilon_q$ bounds $\|u_{h,q}-\widetilde u_q\|_a$ from the exact interior
linear-system residual.  The triangle inequality gives
$\widehat\delta_q$, and the argument for $r$ is identical.  Finally,
$\widehat\eta_{qr}$ bounds $\|Z_{qr}^h-\widetilde z\|_a$ after transferring
all three discrete-data errors.  Since $\widetilde z\in V_h$, Galerkin best
approximation yields
\[
 \|Z_{qr}^h-z_h\|_a
 \leq\|Z_{qr}^h-\widetilde z\|_a\leq\widehat\eta_{qr}.
\]
The lifting residual enters $E_{\rm alg}$ through
$\|r_{\rho_z}\|_2\epsilon_0/\sqrt{\underline\alpha}$ in
\eqref{eq:center-algebraic-bound}.  Therefore
\begin{equation}\label{eq:combined-BJ}
 B_J=\widehat\delta_q\widehat\delta_r
 +\frac12C_{qr}\widehat\delta_0^2
 +\widehat\delta_0\widehat\eta_{qr}
\end{equation}
is a continuous discretization radius for $J_{qr}-J_{h,qr}$.

Let $A_q=DA(P_n)[q]$ and
$A_{qr}=D^2A(P_n)[q,r]$ denote derivatives of polygonal area, as
distinct from the coefficient matrices \eqref{eq:Aq-pullback}--
\eqref{eq:Aqr-pullback}.  For the nonzero Fourier modes used below,
$A_q=A_r=0$ exactly.  Combining
Propositions~\ref{prop:corrected-centers} and~\ref{prop:second-residual}, the
center
\begin{equation}\label{eq:corrected-F-center}
 \widetilde F_{qr}=\frac{\widetilde C_{\rm corr}}{A^2}
       -\frac{2\widetilde J_{\rm corr}A_{qr}}{A^3}
\end{equation}
approximates the continuous scale-Hessian entry with the guaranteed error
\begin{equation}\label{eq:complete-F-radius}
 R_F=\frac{B_J+E_{\rm alg}}{A^2}
 +\frac{|A_{qr}|}{A^3}
       \bigl(\widehat\delta_0^2+\epsilon_0^2\bigr).
\end{equation}
Thus the conclusion for the chosen directions is
\[
 \big|D^2\F(P_n)[q,r]-\widetilde F_{qr}\big|\leq R_F.
\]
The first term of \eqref{eq:complete-F-radius} transfers the continuous and
algebraic errors in $J_{qr}$ through the factor $A^{-2}$.  The second
transfers the error in $J$ through the area correction.  The polygonal area
and its derivatives come directly from the vertices.

Every quantity in these formulas is evaluated with outward rounding.
If the computed center is itself a ball, its radius is retained when
enlarging it by $R_F$.  This accounts for arithmetic error in the same
enclosure as the PDE and linear-system errors.

Returning to $\alpha_2$ in the pentagon, the corrected enclosure is
contained in the stored-center interval displayed in
Section~\ref{sec:certificate-target}.  Its radius is about $4.74\,10^{-4}$,
while the magnitude of the center is about $1.47\,10^{-2}$.
The continuous discretization error accounts for almost all of this
radius; the algebraic contribution is below $1.1\,10^{-25}$.
To complete the example, take the complete pentagon block $B_2$ in
\eqref{eq:symbol-entries}.  Rounding the stored centers to the digits below
and using the common outward radius $4.9\,10^{-4}$ gives
\[
\begin{aligned}
 |\alpha_2-(-0.0147312069261)|&\leq4.9\,10^{-4},\\
 |\beta_2-(-0.00994311747728)|&\leq4.9\,10^{-4},\\
 |\gamma_2-0.00152976372116|&\leq4.9\,10^{-4}.
\end{aligned}
\]
The real part of the off-diagonal is zero by reflection; retaining an
interval of radius $4.9\,10^{-4}$ around zero also gives a valid enclosure.
Using all four real entry intervals in the Hermitian eigenvalue formula
and rounding the resulting endpoints outward yields
\[
 \mu_2^-\in[-0.0164,-0.0137],\qquad
 \mu_2^+\in[-0.0110,-0.0082].
\]
Both upper endpoints are strictly negative.  Conjugacy gives the same two
enclosures for $B_3$, accounting for four negative eigenvalues.  The
negative trace in each of $B_1,B_4$ supplies the other two, while the
symmetry argument supplies the four exact zeros.  Section~\ref{sec:pentagon}
gives the complete pentagon certificate.

The entry radii are computed from the residual bounds before the spectral
sign test.  They are not chosen as fractions of the computed eigenvalues.
They vary with the mode because the displacement fields and their
error bounds vary with the mode; the eigenvalue formula then propagates
these entry uncertainties.  The corrected centers follow the fixed
formula \eqref{eq:corrected-F-center}, independently of the outcome of the
sign test.  A branch is certified negative only if its entire final
interval lies below zero.

Similarity is used exactly: $B_0=0$ and the translation branch in each of
$k=1,n-1$ is zero.  Only the trace must be certified negative in these
translation modes.  Conjugate modes share their diagonal entries and have
opposite off-diagonal imaginary parts, as in
\eqref{eq:conjugate-symbols}.

\subsection{How the implementation verifies the hypotheses}\label{sec:certificate-implementation}

FreeFEM~\cite{Hecht2012} computes candidate fields and flux potentials.
The verifier uses FLINT/Arb~\cite{Hart2010,Johansson2017} to enclose the
residuals and every subsequent arithmetic operation; see
\cite{Rump2010} for the principles of outward verification.
The checks follow the mathematical argument:
\begin{enumerate}
\item Reconstruct the exact regular fan and normalized Fourier directions,
verify the supplied mesh incidence and coordinate containment, and impose
the exact zero boundary coefficients.  This identifies the domain, space,
and directions to which the certificate applies.
\item Assemble the four interior systems, bound their residuals, and
integrate the flux mismatches outwardly.  These calculations supply
\eqref{eq:algebraic-transfer} and \eqref{eq:combined-deltas}.
\item Evaluate the corrected center and full radius
\eqref{eq:corrected-F-center}--\eqref{eq:complete-F-radius}.  Prove that
this entry enclosure lies in the stored-center ball passed to the
eigenvalue calculation, including the conjugate-mode counterpart.
\item Evaluate the eigenvalue intervals, insert the exact similarity
zeros, and check the complete count of $2n-4$ negative branches and four
zeros.
\end{enumerate}
The entry verifier is \texttt{second\_lifting\_cert.c}; the final symbol
test is \texttt{mode\_cert.c}.  The drivers use the same declared radii and
192-bit arithmetic for entry containment and eigenvalue evaluation.
Section~\ref{sec:reproducibility} gives the commands, source map, and
archived outputs.  Appendix~\ref{app:assembly} records the element and
connectivity details needed to audit the assembly.

\section{Finite-element rate: regularity, proof, and experiment}
\label{sec:rate}

We prove quadratic convergence of the exact Galerkin Hessian at each
fixed regular polygon.  The main issue is regularity at the endpoints of
the fan rays. First material derivatives have gradient jumps across the
rays, and second derivatives require a compatibility argument at the center.
The proof follows the ray decomposition of
Bogosel--Bucur~\cite[Section~3]{BogoselBucurValidated2024}, with a positive
Sobolev shift and an explicit lifting of the second-order center jumps.
All constants in this section may depend on the fixed polygon and mesh
shape regularity.  No uniformity in $n$ or in a neighborhood of perturbed
polygons is asserted.

\subsection{A ray source and its endpoint compatibility}

Write $\Gamma_j=[0,a_j]$ for a coarse fan ray.  Choose
\begin{equation}\label{eq:regularity-shift}
 0<\eta<\min\left\{\frac12,\frac{2}{n-2}\right\}.
\end{equation}
The Dirichlet corner exponent of $P_n$ is $n/(n-2)$.
Thus the polygonal Dirichlet shift theorem
\cite[Chapter~4]{Grisvard1985} maps an $H^\eta(P_n)$ right-hand side
to an $H^{2+\eta}(P_n)$ solution.  The strict inequality in
\eqref{eq:regularity-shift} stays below the first corner singularity.
We also use two standard fractional Sobolev extension facts: extension
by zero across a Lipschitz interface preserves $H^\eta$ for
$0<\eta<1/2$, and a function in $H^{1/2+\eta}$ on an interval extends
by zero in that space when both endpoint traces vanish.

\begin{lemma}[a ray source with zero endpoint values]\label{lem:ray-shift}
Let $g\in H^{1/2+\eta}(\Gamma_j)$ satisfy $g(0)=g(a_j)=0$.
The weak solution $w\in H^1_0(P_n)$ of
\[
 a(w,v)=\int_{\Gamma_j}gv\qquad(v\in H^1_0(P_n))
\]
belongs to $H^{2+\eta}(T_i)$ on every coarse fan triangle, with
\[
 \sum_i\|w\|_{H^{2+\eta}(T_i)}
 \leq C\|g\|_{H^{1/2+\eta}(\Gamma_j)}.
\]
\end{lemma}
\begin{proof}
The trace theorem makes the load continuous on $H^1_0(P_n)$.
Reflection in the line containing $\Gamma_j$ preserves the polygon and
the load, so uniqueness makes $w$ even.  On either half-polygon it solves
the mixed problem with zero Dirichlet data on the outer sides, outward
Neumann data $g/2$ on $\Gamma_j$, and zero Neumann data on the rest of
the cut.  The factor $1/2$ follows by doubling the weak integral on one
half-polygon.  Zero extension through the center gives
$H^{1/2+\eta}$ data on the whole cut.

Here is a reduction that verifies the positive shift, including the
mixed corners.  Prescribe a lifting $v_0$ with zero value on the entire
boundary of the half-polygon, normal derivative $g/2$ on the loaded cut,
and zero normal derivative on all remaining sides.  The polygonal trace
lifting theorem of Bernardi--Dauge--Maday
\cite[Theorem~1]{BernardiDaugeMaday2000}, with $p=2$, $s=2+\eta$, and
two traces, gives $v_0\in H^{2+\eta}$ with a bounded norm.  Its corner
conditions require agreement of the value and gradient: the prescribed
value is zero, and both one-sided gradients are zero at each corner
because the Neumann data vanish at the cut endpoints.  There is no
second-derivative compatibility condition since $2+\eta<5/2$.

Now $w-v_0$ has homogeneous Neumann data on the cut and homogeneous
Dirichlet data on the outer sides.  Its even reflection is a weak
Dirichlet solution on $P_n$, with right-hand side equal to the even
reflection of $\Delta v_0$.  This right-hand side belongs to $H^\eta$
because $\eta<1/2$.  The Dirichlet shift above gives $H^{2+\eta}$
regularity on $P_n$.  Adding $v_0$ proves the asserted regularity on
each half-polygon.

For even $n$, both endpoints of the cut are polygon vertices.  For odd
$n$, the other endpoint is an edge midpoint; the data are zero near it,
and the same reflection gives a straight Dirichlet boundary there.
The negative part of the cut can pass through a fan triangle when $n$
is odd.  It introduces no additional interface: $w$ is harmonic across
it away from the center, so its first-derivative traces agree.  The
piecewise $H^\eta$ second derivatives then glue across that cut for
$\eta<1/2$.  This proves $H^{2+\eta}$ regularity on each original fan
triangle as well as the norm bound.
\end{proof}

\subsection{First and second material derivatives}

\begin{proposition}[broken regularity at the regular polygon]
\label{prop:broken-regularity}
For fixed $n\geq5$, the state satisfies $u\in H^{2+\eta}(P_n)$.
For any vertex directions $q,r$, its exact material derivatives satisfy
\[
 u_q|_{T_j}\in H^{2+\eta}(T_j),\qquad
 u_{qr}|_{T_j}\in H^2(T_j).
\]
Their broken norms are bounded uniformly for $|q|,|r|\leq1$.
The first-derivative gradients have a common value at the center and
vanish at every polygon vertex.
\end{proposition}
\begin{proof}
The pulled-back forms \eqref{eq:pulled-forms} depend smoothly on the
parameters and remain coercive for small parameters.  Differentiating
their bounded inverse gives $u_q,u_{qr}\in H^1_0(P_n)$ and their weak
equations.  We establish their spatial regularity only at the reference
polygon.

\emph{State and first derivatives.}
The Dirichlet shift gives $u\in H^{2+\eta}(P_n)$, so $\nabla u$ is
continuous up to the vertices.  Its tangential derivatives vanish on
the two adjacent boundary sides, giving $\nabla u(a_j)=0$.
Rotational symmetry gives $\nabla u(0)=0$.

Let $\nu_j$ point from $T_{j-1}$ into $T_j$ and write
$[f]_j=f|_{T_j}-f|_{T_{j-1}}$.  Continuity of the affine fan extension
implies, for a constant vector $b_{q,j}$,
\[
 [D\theta_q]_j=b_{q,j}\otimes\nu_j,\qquad
 \nu_j\cdot[\mathbb A_q]_j p=-b_{q,j}\cdot p
 \quad(p\in\R^2),
\]
where
$\mathbb A_q=(\operatorname{div}\theta_q)I-D\theta_q-D\theta_q^T$.
Integrating its weak equation by parts sectorwise gives the bulk
right-hand side and the required normal-derivative jump
\begin{equation}\label{eq:first-transmission}
 \begin{aligned}
 -\Delta u_q&=\operatorname{div}\theta_q+\mathbb A_q:D^2u
                   &&\text{in }T_j,\\
 [\partial_{\nu_j}u_q]_j
     &=-\nu_j\cdot[\mathbb A_q\nabla u]_j
       =b_{q,j}\cdot\nabla u &&\text{on }\Gamma_j.
 \end{aligned}
\end{equation}
Here and below a jump first specifies the ray load in the weak equation;
we do not assume the regularity being proved.  With this orientation,
a jump $g$ contributes the distribution $-g\delta_{\Gamma_j}$ to
$-\Delta u_q$, where $g\delta_{\Gamma_j}$ acts by
$v\mapsto\int_{\Gamma_j}gv$.

The bulk term is piecewise $H^\eta$, hence globally $H^\eta$ since
$\eta<1/2$.  Each jump density is $H^{1/2+\eta}$ and vanishes at both
endpoints.  Decompose the weak solution into the Dirichlet bulk solution
and the finitely many ray-source solutions.  The Dirichlet shift and
Lemma~\ref{lem:ray-shift} give broken $H^{2+\eta}$ regularity.
Its one-sided gradients are therefore continuous to sector endpoints.
Tangential derivatives agree across each ray, and the normal jump in
\eqref{eq:first-transmission} vanishes at both endpoints.  Thus all
sector gradients have the same value at the center.  At each outer
vertex, the two sector gradients also agree; the two Dirichlet sides
force this common gradient to be zero.  In general $\nabla u_q(0)$
need not vanish.

\emph{Second derivatives and the center lifting.}
The exact second derivative solves
\[
 a(u_{qr},v)=\ell_{qr}(v)-a_q(u_r,v)-a_r(u_q,v)-a_{qr}(u,v).
\]
Its bulk right-hand side is
\[
 (\operatorname{div}\theta_q)(\operatorname{div}\theta_r)
 -\tr(D\theta_qD\theta_r)
 +\mathbb A_q:D^2u_r+\mathbb A_r:D^2u_q+\mathbb A_{qr}:D^2u,
\]
which is $L^2$ sectorwise and therefore globally $L^2$.
The required normal jump on $\Gamma_j$ is
\begin{equation}\label{eq:second-transmission}
 -\nu_j\cdot[\mathbb A_q\nabla u_r+
             \mathbb A_r\nabla u_q+\mathbb A_{qr}\nabla u]_j.
\end{equation}
It belongs to $H^{1/2+\eta}$ and vanishes at $a_j$.  At the center its
value is
$b_{q,j}\cdot\nabla u_r(0)+b_{r,j}\cdot\nabla u_q(0)$.
These values need not vanish separately, so the ray lemma cannot yet be
applied to the individual rays.

Choose a smooth cutoff $\chi$, supported in an interior disk and equal
to one near the center, and subtract the function
\begin{equation}\label{eq:center-jump-lifting}
 L(x)=\chi(x)\bigl(\theta_q(x)\cdot\nabla u_r(0)
                    +\theta_r(x)\cdot\nabla u_q(0)\bigr).
\end{equation}
It is continuous, belongs to $H^1_0(P_n)$, and is smooth on each sector.
Near the center its normal jump is exactly
$b_{q,j}\cdot\nabla u_r(0)+b_{r,j}\cdot\nabla u_q(0)$.
Consequently $u_{qr}-L$ has ray densities in $H^{1/2+\eta}$ that
vanish at both endpoints, and an $L^2$ bulk right-hand side.  Convex
Dirichlet $H^2$ regularity handles the bulk solution, while
Lemma~\ref{lem:ray-shift} handles each ray solution.  Their sum equals
$u_{qr}-L$ by weak uniqueness, proving broken $H^2$ regularity of
$u_{qr}$.  Finally, the dependence on $q$ is linear and on $(q,r)$
bilinear in finite dimension.  Applying the preceding bounds to the
coordinate directions gives uniform bounds for $|q|,|r|\leq1$.
\end{proof}

\subsection{Quadratic convergence of the Hessian}

On shape-regular conforming $P_1$ refinements of the coarse fan, broken
interpolation and the differentiated Galerkin equations now give
\begin{equation}\label{eq:w-rates}
 \|\nabla(u-u_h)\|=O(h),\qquad \|u-u_h\|=O(h^2),
\end{equation}
and
\begin{equation}\label{eq:U-rate}
 \|\nabla(U_i^p-U_{i,h}^p)\|=O(h).
\end{equation}
Here $h$ is the largest element diameter.  For the state alone, alignment
with the fan rays is unnecessary.  For material derivatives it matters:
gradient jumps lie along element edges on a fitted mesh, allowing
interpolation sector by sector.  No unfitted mesh estimate is claimed.

\begin{theorem}[quadratic Hessian rate at the regular polygon]\label{prop:h2}
For each fixed $n\geq5$, exact $P_1$ Galerkin solutions on shape-regular
conforming refinements fitted to the coarse fan satisfy
\[
 \|D^2J(P_n)-D^2J_h(P_n)\|_2=O(h^2),\qquad
 \|D^2\F(P_n)-D^2\F_h(P_n)\|_2=O(h^2).
\]
Here $\F_h=J_h/A^2$, with the exact polygon area.
The ordered Hessian eigenvalues satisfy the same error bound.  On the
symmetric uniform fan subdivisions used in this paper, $h$ is
proportional to $m^{-1}$, the discrete Hessian has the same Fourier
block structure as the continuous one, and
\begin{equation}\label{eq:h2-rate}
 \|B_k-B_{k,h}\|_2=O(h^2),\qquad
 |\mu_k^\pm-\mu_{k,h}^\pm|=O(h^2).
\end{equation}
This includes the continuous similarity-zero branches when their discrete
values are not imposed algebraically.
\end{theorem}
\begin{proof}
In this proof, take $\delta_0,\delta_q,\delta_r,\eta_{qr}$ to be the
actual energy errors in Proposition~\ref{prop:second-residual}.
Broken interpolation and C\'ea's lemma give $\delta_0=O(h)$.
The first differentiated Galerkin equation uses $u_h$ in its load, so it
is not the Ritz projection of $u_q$.  Subtracting the continuous equation,
using \eqref{eq:Cqr} and Proposition~\ref{prop:broken-regularity}, gives
\[
 \delta_q\leq C\left(\inf_{v_h\in V_h}\|u_q-v_h\|_a+\delta_0\right)=O(h),
 \qquad \delta_r=O(h).
\]
Subtract the exact second-derivative equation from
\eqref{eq:second-lifting} to obtain
\[
 \|Z_{qr}^h-u_{qr}\|_a
 \leq C_q\delta_r+C_r\delta_q+C_{qr}\delta_0=O(h).
\]
Since $z_h$ is the Galerkin projection of $Z_{qr}^h$, a fitted nodal
interpolant $I_hu_{qr}$ gives
\[
 \eta_{qr}\leq\|Z_{qr}^h-I_hu_{qr}\|_a
 \leq\|Z_{qr}^h-u_{qr}\|_a+\|u_{qr}-I_hu_{qr}\|_a=O(h).
\]
Thus no broken $H^2$ estimate for the $h$-dependent lifting $Z_{qr}^h$
is needed.  Every term in \eqref{eq:second-residual-bound} is $O(h^2)$,
uniformly over unit vertex directions.  This proves the operator-norm
bound for $D^2J$.

The energy identity \eqref{eq:energy-error-identity} and its first
derivative give
\[
 J-J_h=\tfrac12a(e,e)=O(h^2),\qquad
 (J-J_h)_q=\tfrac12a_q(e,e)+a(e,e_q)=O(h^2).
\]
Differentiating the exact area factor $A^{-2}$ therefore gives the rate
for $D^2\F$.  Weyl's inequality gives the ordered eigenvalue bound.
On a symmetry-preserving mesh, restriction of the unitary Fourier
transform to each block gives \eqref{eq:h2-rate}.
\end{proof}

This theorem concerns the exact Galerkin Hessian at $P_n$.  It does not
assert uniform broken regularity throughout a neighborhood of vertex
coordinates.  An $O(h^2)$ rate for the computed enclosure radius
\eqref{eq:complete-F-radius} additionally requires approximation bounds
for the chosen fluxes and control of algebraic errors; that rate is not
proved here.  The fixed-mesh certificates use the finite residual bounds
of Section~\ref{sec:validated} independently of this convergence theorem.

\subsection{Observed rates}

For the pentagon, three dyadic meshes $m=16,32,64$ give the observed orders
in Table~\ref{tab:rates}.  The order is computed from
\begin{equation}
 p_h=\log_2\left|\frac{\mu_h-\mu_{h/2}}
                            {\mu_{h/2}-\mu_{h/4}}\right|,
 \qquad h=\frac1{16}.
\end{equation}
\begin{table}[ht]
\centering
\begin{tabular}{cccc}
\toprule
mode & branch & value at $m=64$ & observed order\\
\midrule
$1,4$ & nonsimilarity & $-3.5562221538\,10^{-3}$ & $2.005$\\
$1,4$ & translation (zero) & $-2.8252056067\,10^{-6}$ & $1.962$\\
$2,3$ & lower & $-1.5185488157\,10^{-2}$ & $2.227$\\
$2,3$ & upper & $-9.5016877690\,10^{-3}$ & $2.065$\\
\bottomrule
\end{tabular}
\caption{Dyadic rate estimates for the regular pentagon using $m=16,32,64$.}
\label{tab:rates}
\end{table}
The data are consistent with the rate proved in Theorem~\ref{prop:h2}.
They do not replace the continuous residual bound used for certification
in Section~\ref{sec:validated}.

The residual product bound also exhibits nearly quadratic decay in this
experiment.  For
the $k=2$ normalized radial cosine direction, the floating-point evaluations of
the right-hand side of \eqref{eq:second-residual-bound} at
$m=8,16,32,64$ are respectively
\begin{equation}\label{eq:second-bound-rate-data}
 3.84413\,10^{-2},\quad 1.01047\,10^{-2},\quad
 2.59696\,10^{-3},\quad 6.60922\,10^{-4}.
\end{equation}
The three successive decay exponents are $1.93$, $1.96$, and $1.97$.
These are nonvalidated convergence diagnostics; the sign certificates below
instead use outward residual majorants on each fixed mesh and do not use
an extrapolated rate.  The raw Hessian outputs, commands, timings,
and order estimates for Table~\ref{tab:rates} are archived in
\path{MaxTorsionValidation/results/pentagon_rate_m16_m32_m64/}.
The residual experiment and the earlier $m=8,16,32$ Hessian comparison
are recorded in \path{MaxTorsionValidation/results/pentagon_rate_experiment.txt}.

\section{Pentagon benchmark and validated sign enclosure}\label{sec:pentagon}

At $n=5$, $m=32$, the floating-point FreeFEM computation gives
\begin{equation}
J_h=0.1055228910051348,\qquad A=2.377641290737884.
\end{equation}
The six nonsimilarity eigenvalues of the discrete scale-invariant Hessian
are
\begin{align}\label{eq:n5-eigs}
 -0.0035602324903&\quad(\text{multiplicity }2),\notag\\
 -0.0151782231277&\quad(2),&
 -0.0094961012757&\quad(2).
\end{align}
The discrete translation branch is $-1.1208786894\,10^{-5}$ instead of its
exact continuous value zero; its decay is a useful discretization
diagnostic.  None of these floating-point values is used without a residual
enclosure in the proof below.

For each $k=1,2$, use normalized real cosine radial and tangential directions.
The radial--radial, tangential--tangential, and radial cosine--tangential
sine pairings give $\alpha_k,\beta_k$, and $\gamma_k$.  The radial
cosine--tangential cosine pairing vanishes by reflection symmetry and is
retained in the implementation as a consistency check.  Thus eight entry calculations using
Propositions~\ref{prop:corrected-centers} and
\ref{prop:second-residual} enclose all entries, including the two real
off-diagonal entries known to vanish by symmetry; six runs would suffice for
the independent symbol entries.  Conjugacy gives $k=4,3$.

The input coordinates and direction coefficients are placed in balls of
radius $10^{-13}$.  Before integration, Arb reconstructs every exact lattice
point of the regular fan and every exact normalized Fourier direction and
checks containment.  The state, first-variation, and lifting fluxes are
exactly equilibrated by \eqref{eq:curl-state-flux},
\eqref{eq:curl-material-flux}, and
\eqref{eq:second-source-divergence}.  The finite-system residuals are also
assembled outwardly on the exact interior submatrices.  At $m=32$, every
transferred algebraic energy error is below $5.5\,10^{-13}$, and the largest
algebraic contribution to a normalized Hessian-entry radius is below
$1.1\,10^{-25}$.  The finite-element discretization error, not the floating-point
linear solve, therefore determines the displayed radii.

Table~\ref{tab:pentagon-entry-radii} lists the eight total scale-Hessian
entry radii.

\begin{table}[ht]
\centering
\begin{tabular}{cccc}
\toprule
entry&Exact value&radius ($k=1$)&radius ($k=2$)\\
\midrule
$rr$&$\alpha_k$&$1.748872\,10^{-4}$&$4.743372\,10^{-4}$\\
$tt$&$\beta_k$&$1.392985\,10^{-4}$&$2.863460\,10^{-4}$\\
$rt,\operatorname{Re}$&$0$&$1.329830\,10^{-4}$&$2.949722\,10^{-4}$\\
$rt,\operatorname{Im}$&$\gamma_k$&$1.413011\,10^{-4}$&$2.930991\,10^{-4}$\\
\bottomrule
\end{tabular}
\caption{Outward continuous-plus-algebraic radii for the
pentagon symbol entries.}
\label{tab:pentagon-entry-radii}
\end{table}

The largest radius belongs to the radial diagonal entry in mode $k=2$:
\begin{equation}\label{eq:largest-entry-certificate}
 \widetilde B_{2}^{11}=-0.0147312069261,\qquad
 |B_2^{11}-\widetilde B_{2}^{11}|\leq4.743372\,10^{-4}.
\end{equation}
Taking an upward-rounded common bound for all eight radii gives $4.9\,10^{-4}$.
For each entry the driver machine-checks containment of both conjugate modes
in the corresponding stored-center ball of that radius.  Arb then proves
six strictly negative eigenvalue intervals.  The executable inserts the four
similarity zeros from Proposition~\ref{prop:modes} and asserts the complete
count:
\begin{equation}\label{eq:validated-count}
 \#\{\hbox{certified negative eigenvalues}\}=6=2n-4,
 \qquad \#\{\hbox{exact similarity zeros}\}=4.
\end{equation}

\begin{theorem}[computer-assisted pentagon result]\label{thm:pentagon}
The regular pentagon is a strict local maximizer of torsional rigidity among
nearby pentagons of the same area, modulo translations and rotations.
Equivalently, it is a strict local maximizer of $J/A^2$ modulo similarities.
\end{theorem}
\begin{proof}
The exact symmetry argument in Proposition~\ref{prop:modes} supplies the four
zero modes.  The outward residual and eigenvalue calculation above proves
strict negativity of the other six modes.  The conclusion follows from
Theorem~\ref{thm:conditional}.  The convergence theorem is not used here:
all continuous errors are bounded a posteriori on the single fixed mesh
$m=32$.
\end{proof}

\section{Validated certificates for regular polygons}
\label{sec:n3-n10}

The pentagon construction extends without changing the analytic error
identity.  A single global entry radius becomes wasteful as $n$ grows, so we
use a radius $\rho_{n,k}$ shared only by the conjugate pair $k,n-k$.  For odd
$n$ this requires $2(n-1)$ real lifting calculations.  For even $n$ it
requires $2n-2$, because reflection makes the Nyquist off-diagonal exactly
zero.  The $m=32$ run in Table~\ref{tab:n3-n10-certificates} validates
88 real liftings.

\begin{table}[ht]
\centering
\small
\begin{tabular}{rclrrc}
\toprule
$n$ & $m$ & $\rho_{n,k}$, $1\leq k\leq\lfloor n/2\rfloor$
    & liftings & negative & exact zeros\\
\midrule
3  &32&$1.32\,10^{-3}$&4&2&4\\
4  &32&$2.60\,10^{-4},\ 5.70\,10^{-4}$&6&4&4\\
5  &32&$1.80\,10^{-4},\ 4.80\,10^{-4}$&8&6&4\\
6  &32&$1.30\,10^{-4},\ 3.30\,10^{-4},\ 4.10\,10^{-4}$&10&8&4\\
7  &32&$9.60\,10^{-5},\ 2.80\,10^{-4},\ 5.20\,10^{-4}$&12&10&4\\
8  &32&$7.50\,10^{-5},\ 1.85\,10^{-4},\ 4.65\,10^{-4},\ 4.80\,10^{-4}$&14&12&4\\
9  &32&$6.40\,10^{-5},\ 1.90\,10^{-4},\ 4.10\,10^{-4},\ 6.60\,10^{-4}$&16&14&4\\
10 &32&$5.50\,10^{-5},\ 1.60\,10^{-4},\ 4.00\,10^{-4},\ 6.50\,10^{-4},\ 6.20\,10^{-4}$&18&16&4\\
\bottomrule
\end{tabular}
\caption{Modewise continuous-plus-algebraic entry radii and certified Arb
counts.  Conjugate modes use the same radius.}
\label{tab:n3-n10-certificates}
\end{table}

For every row, the executable first checks the exact fan topology and exact
regular geometry, then proves containment of every entry and its conjugate
in the benchmark center plus $\rho_{n,k}$.  The final count is performed on
all $2n$ eigenvalues.  The four zeros are not inferred from small numerical
intervals: they are the two $k=0$ similarity directions and the two
translations inserted from Proposition~\ref{prop:modes}.  For completeness,
Table~\ref{tab:n3-n10-margins} gives an outward-rounded upper
bound $U_n<0$ for every nonsimilarity eigenvalue.  The limiting branch is the
nonzero $k=1$ trace eigenvalue in every row.

\begin{table}[ht]
\centering
\begin{tabular}{rc@{\qquad}rc}
\toprule
$n$ & $U_n$ & $n$ & $U_n$\\
\midrule
3 &$-1.00\,10^{-2}$&7 &$-9.80\,10^{-4}$\\
4 &$-6.20\,10^{-3}$&8 &$-5.90\,10^{-4}$\\
5 &$-3.10\,10^{-3}$&9 &$-3.60\,10^{-4}$\\
6 &$-1.70\,10^{-3}$&10&$-2.20\,10^{-4}$\\
\bottomrule
\end{tabular}
\caption{Certified continuous sign margins: every nonsimilarity Hessian
eigenvalue $\lambda$ satisfies $\lambda\leq U_n<0$.}
\label{tab:n3-n10-margins}
\end{table}

We next refine the same coarse fan, keeping the residual identity and
the 192-bit Arb verification unchanged.  Table~\ref{tab:m64-certificates}
reports the complete tests for $11\leq n\leq20$ at $m=64$.  Every entry and
its conjugate is enclosed before the final eigenvalue test.  The upper-bound
column gives an outward-rounded bound for every nonsimilarity eigenvalue;
a negative value proves strict negativity on the complement of the
similarity directions.
The four exact similarity zeros are excluded from the counts.

\begin{table}[ht]
\centering
\small
\begin{tabular}{rrrrl}
\toprule
$n$ & triangles & negative / required & upper bound & result\\
\midrule
11 & $45\,056$ & 18 / 18 & $-2.11\,10^{-4}$ & certified\\
12 & $49\,152$ & 20 / 20 & $-1.50\,10^{-4}$ & certified\\
13 & $53\,248$ & 22 / 22 & $-1.08\,10^{-4}$ & certified\\
14 & $57\,344$ & 24 / 24 & $-7.99\,10^{-5}$ & certified\\
15 & $61\,440$ & 26 / 26 & $-5.97\,10^{-5}$ & certified\\
16 & $65\,536$ & 28 / 28 & $-4.51\,10^{-5}$ & certified\\
17 & $69\,632$ & 30 / 30 & $-3.44\,10^{-5}$ & certified\\
18 & $73\,728$ & 32 / 32 & $-2.64\,10^{-5}$ & certified\\
19 & $77\,824$ & 30 / 34 & $ 1.35\,10^{-5}$ & inconclusive\\
20 & $81\,920$ & 28 / 36 & $ 1.69\,10^{-4}$ & inconclusive\\
\bottomrule
\end{tabular}
\caption{Validated computations at $m=64$, with $nm^2$ triangles.
The required number of negative eigenvalues is $2n-4$.}
\label{tab:m64-certificates}
\end{table}

At $m=64$, all polygons through $n=18$ in this test range are certified.
For $n=19$, the upper branches in modes $k=6,7$ and their conjugates have
enclosures containing zero.  For $n=20$, this happens in modes $k=6,7,8,9$
and their conjugates.  These inconclusive results concern the present
enclosures and eigenvalue calculation; they do not disprove local
maximality or rule out certification with sharper bounds.

At $m=128$, the complete tests for $19\leq n\leq25$ are reported
in Table~\ref{tab:m128-certificates}.  All polygons through $n=25$
in this range are certified.  In particular, this refinement resolves
the inconclusive $m=64$ tests for $n=19,20$.
The theorem and tables in this paper are restricted to $n\leq25$.
The same procedure applies to larger polygons; a subsequent $n=26$ run
at $m=128$ also completed successfully.  Extending the range further
requires sufficient computing time and memory and, where necessary,
finer meshes to resolve the sign.  No certification limit at this
refinement level is asserted.

\begin{table}[ht]
\centering
\small
\begin{tabular}{rrrrl}
\toprule
$n$ & triangles & negative / required & upper bound & result\\
\midrule
19 & $311\,296$ & 34 / 34 & $-2.73\,10^{-5}$ & certified\\
20 & $327\,680$ & 36 / 36 & $-2.22\,10^{-5}$ & certified\\
21 & $344\,064$ & 38 / 38 & $-1.81\,10^{-5}$ & certified\\
22 & $360\,448$ & 40 / 40 & $-1.49\,10^{-5}$ & certified\\
23 & $376\,832$ & 42 / 42 & $-1.24\,10^{-5}$ & certified\\
24 & $393\,216$ & 44 / 44 & $-1.03\,10^{-5}$ & certified\\
25 & $409\,600$ & 46 / 46 & $-8.71\,10^{-6}$ & certified\\
\bottomrule
\end{tabular}
\caption{Validated computations at $m=128$.  As in
Table~\ref{tab:m64-certificates}, the upper bound applies to every
nonsimilarity eigenvalue and the four exact zeros are excluded.}
\label{tab:m128-certificates}
\end{table}

\FloatBarrier
\begin{proof}[Proof of Theorem~\ref{thm:n3-n10}]
Proposition~\ref{prop:modes} gives the four exact similarity directions in the kernel.  The
outward Arb computations summarized in Tables~\ref{tab:n3-n10-certificates},
\ref{tab:m64-certificates}, and~\ref{tab:m128-certificates} prove strict
negativity of every other eigenvalue, using $m=32$ for $5\leq n\leq10$,
$m=64$ for $11\leq n\leq18$, and $m=128$ for $19\leq n\leq25$.  Apply
Theorem~\ref{thm:conditional}.  No asymptotic rate is used.
\end{proof}

The same certified calculation gives the corresponding inertia for
$n=3,4$, independently recovering the classical local conclusions.

\section{Reproducibility, audit trail, and trusted computing base}
\label{sec:reproducibility}

All paths and commands below are relative to the repository root, which
contains \path{MaxTorsionValidation/}.  The maintained code directory was
formerly named \path{code/}; historical archive manifests retain that old
name.  The computation uses the following files:
\begin{itemize}
\item \path{MaxTorsionValidation/freefem/torsion_hessian.edp}: distributed Hessian,
scale normalization, radial basis, and Fourier symbols;
\item \path{MaxTorsionValidation/freefem/analyze_convergence.py}: dyadic rate estimator;
\item \path{MaxTorsionValidation/freefem/second_variation_lifting.edp}: modewise
first variations, second lifting, and exactly equilibrated curl candidates;
\item \path{MaxTorsionValidation/freefem/run_pentagon_certificate.sh}: the eight-entry
pentagon certificate and final eigenvalue count;
\item \path{MaxTorsionValidation/freefem/run_regular_certificates.sh}: all mode-entry
certificates and final eigenvalue counts for $3\leq n\leq10$;
\item \path{MaxTorsionValidation/freefem/test_extended_certificates.py}: finer-mesh
tests with automatically proposed entry radii, followed by rigorous
containment and complete eigenvalue-count checks;
\item \path{MaxTorsionValidation/freefem/certify_with_rotations.py}: the optimized
continuation, reusing the first-vertex solutions and flux potentials;
\item \path{MaxTorsionValidation/flint/mode_cert.c}: Arb evaluation of the two eigenvalues
of each Hermitian symbol, including per-mode radii and the exact regular
similarity reductions;
\item \path{MaxTorsionValidation/flint/residual_cert.c}: Arb residual certificate
\eqref{eq:algebraic-residual} for SPD PDE systems;
\item \path{MaxTorsionValidation/flint/second_lifting_cert.c}: exact regular-fan input
check, Arb triangle assembly, continuous and algebraic residuals, and the
scale-Hessian entry enclosures used in all the reported certificates.
\end{itemize}

The source code is available at
\begin{center}
\href{https://github.com/beniamin-bogosel/MaxTorsionValidation}{\nolinkurl{https://github.com/beniamin-bogosel/MaxTorsionValidation}}.
\end{center}
Setup instructions and commands for every table are also given in
\path{MaxTorsionValidation/docs/REPRODUCING_SECTION_8.md}.
Reproduction from these sources generates the candidate fields and checks
the full certificate, without requiring a download of the large historical
archives.  As a small starting case, a reader can regenerate and certify
the pentagon using
\begin{verbatim}
python3 MaxTorsionValidation/freefem/certify_with_rotations.py \
  5 5 --m 32 --jobs 2 --cpu-cores 2 --entry-centers \
  --flux-eps 1e-13 --stop-on-failure \
  --archive /tmp/torsion-pentagon
\end{verbatim}
The output directory must be new.  A successful run records
\texttt{CERTIFIED} in \path{n5/RESULT.json}, eight real entry certificates,
and six negative eigenvalues with four exact similarity zeros in the
final mode log.  This is a fresh certificate; floating-point candidates
and the resulting radii need not be identical to the historical files.

The $m=32$ certificates and their reproducibility archive are generated by
\begin{verbatim}
CERT_ARCHIVE_DIR=/tmp/torsion-certificates \
  bash MaxTorsionValidation/freefem/run_regular_certificates.sh 3 10
\end{verbatim}
The driver refuses an existing archive path and writes \texttt{COMPLETE}
only after all 88 entry containments and all eight final inertia assertions
pass.  The generated archive contains the source files, tool versions,
compressed candidate fields, per-entry logs, reference centers, radii,
final mode logs, and checksums.  The candidate fields allow replay without
a new FreeFEM solve.  The root \texttt{README.md} maps the equation labels
to the source files and output names.

The $m=64$ extension can be reproduced with
\begin{verbatim}
python3 MaxTorsionValidation/freefem/test_extended_certificates.py \
  11 20 --m 64 --jobs 4 --cpu-cores 4 \
  --entry-centers --flux-eps 1e-9 \
  --archive /tmp/torsion-certificates-m64
\end{verbatim}
The reference centers can be taken directly from the lifting calculations;
the proposed radii are accepted only after Arb proves every entry
containment.  The auxiliary curl-fit tolerance affects the sharpness of
the candidate flux, while its admissibility and the resulting residual
bound are checked independently.  The PDE-solve tolerances remain
$10^{-13}$.  Each completed polygon has an explicit result record and a
full eigenvalue log; a finished sweep may include inconclusive sign tests.
The $m=128$ tests for $n=19,20,21$ use
\begin{verbatim}
python3 MaxTorsionValidation/freefem/test_extended_certificates.py \
  19 21 --m 128 --jobs 4 --cpu-cores 4 \
  --entry-centers --flux-eps 1e-9 \
  --archive /tmp/torsion-certificates-m128
\end{verbatim}
Every archive path must be new.
The archived results and source snapshots for Tables~\ref{tab:m64-certificates}
and~\ref{tab:m128-certificates} are indexed in
\path{MaxTorsionValidation/results/extended_m64_summary.md} and
\path{MaxTorsionValidation/results/extended_m128_summary.md}, respectively.
For new runs, the two Python drivers compile private verifiers from the
archived C sources and record executable hashes and compiler/FLINT
identification in the manifest.  They also check that each export and
containment record matches the requested polygon, mode, and directions.

For the optimized continuation, the covariance
\eqref{eq:dihedral-covariance} supplies every first variation from the two
solutions at the first vertex.  The torsion solution and the corresponding
flux potentials are also reused.  Each Fourier pairing still requires its
second lifting and an auxiliary flux fit.  These operations generate
candidates; the exact-input, residual, and containment checks in Arb are
unchanged.  The continuation is run by
\begin{verbatim}
python3 MaxTorsionValidation/freefem/certify_with_rotations.py \
  22 25 --m 128 --jobs 6 --cpu-cores 6 \
  --entry-centers --flux-eps 1e-9 \
  --stop-on-failure --archive /tmp/torsion-rotations-m128
\end{verbatim}
On an Intel Core i7-9750H processor, using six physical cores, the
optimized certificates at $m=128$ took approximately $13$--$20$ minutes
per polygon for $22\leq n\leq25$.  The $n=25$ certificate took
$1059$ seconds (about $17.7$ minutes), including candidate generation,
the Arb entry and eigenvalue checks, and compression of the exported
fields.  Allow roughly $70$ minutes to reproduce the four-polygon range
in the command above on comparable hardware; the recorded times total
about $66$ minutes and depend on the machine load.

The trusted analytic reductions are Laurain's formula and its algebraic
reduction, the residual identities proved above, and the exact similarity,
conjugacy, reflection, and Fourier-normalization identities.  The trusted
software base is the audited C source, its compiler and runtime, and the
outward FLINT/Arb operations.  In particular, the executable's
exact-similarity option is used only after
Propositions~\ref{prop:modes} and~\ref{prop:real-fourier-pairings} establish
its hypotheses; it is not a fact inferred by the executable from arbitrary
input.  FreeFEM's mesh coordinates, linear solves, and symbol centers are
only candidates.  The verifier reconstructs the exact fan, imposes exact
zero boundary coefficients on the interior submatrices, evaluates all
residuals outwardly, and proves that each continuous entry enclosure lies in
the same center--radius ball consumed by the final eigenvalue calculation.
The drivers use 192-bit precision for both the containment check and
the final mode calculation, so these are the same Arb balls.

The test suite is \texttt{make -C MaxTorsionValidation/flint test}.  The implementation also checks
the criticality defect
$\|\nabla J-(2J/A)\nabla A\|_\infty$ and Hessian symmetry.  The symmetry defect is measured before symmetrization; measuring it after
projection would give zero by construction.  The implementation also checks
agreement between the eigenvalues of modes $k$ and $n-k$ and the vanishing
of the $k=0$ symbol up to roundoff.

\section{Conclusions and next rigorous steps}

The Bogosel--Bucur strategy reduces local maximality to $n$ Hermitian
$2\times2$ symbols.  Dihedral symmetry and the constant sector gradients of
the coarse hat functions reduce the local terms to the diagonal expression
\eqref{eq:explicit-cancellation}; only the three Gram sums
\eqref{eq:G-three} remain difficult.  The formulas for $X,Y,Z$ show that the
explicit trace contribution depends only on $J$, while the single
sector anisotropy controls the difference between the explicit diagonal
entries.  Similarity
invariance accounts for four zeros.  The
second-variation identity and the broken regularity proved in
Section~\ref{sec:rate} give an $O(h^2)$ Galerkin Hessian error at each
fixed regular polygon, and
exactly equilibrated curl fluxes and Arb triangle integration certify all
$2n-4$ negative nonsimilarity eigenvalues for every $3\leq n\leq25$.  This proves local
maximality throughout that range, with $n=3,4$ serving as classical
consistency cases and $5\leq n\leq25$ giving Theorem~\ref{thm:n3-n10}.
The analytic local-maximality theorem for every $n\geq5$
\cite{BogoselBucurFragala2026} places these finite-range certificates in a
broader setting.  The numerical contribution remains the rigorous
control of geometric Hessians, including the quadratic Galerkin rate,
and a reproducible certification procedure for discrete geometric
optimization.  Further work includes:
\begin{enumerate}
\item extend the modewise lifting certificate beyond the certified range, with refinement
chosen according to the shrinking gap between nonsimilarity eigenvalues
and zero;
\item establish quantitative approximation bounds for the chosen fluxes
and algebraic tolerances that imply an $O(h^2)$ rate for the computed
certificate radius;
\item extend the regularity analysis to uniform estimates on neighborhoods
of perturbed polygons and apply the certification framework to other
discrete geometric optimization problems.
\end{enumerate}

\appendix
\section{Explicit formulas and assembly details}\label{app:certificate-details}

This appendix records the expanded formulas used by the verifier and the
older comparison estimate.  The enclosure argument is given in
Section~\ref{sec:validated}.

\subsection{Differentiated pullback coefficients}\label{app:pullback}

The coefficient matrices used here and in the finite-element assembly are
\begin{align}
 \mathbb A_q&=(\operatorname{div}\theta_q)I-D\theta_q-(D\theta_q)^T,
 \label{eq:Aq-pullback}\\
 \mathbb A_{qr}
 &=\big[(\operatorname{div}\theta_q)(\operatorname{div}\theta_r)
          -\tr(D\theta_q D\theta_r)\big]I\notag\\
 &\quad-(\operatorname{div}\theta_q)\big(D\theta_r+(D\theta_r)^T\big)
       -(\operatorname{div}\theta_r)\big(D\theta_q+(D\theta_q)^T\big)\notag\\
 &\quad+D\theta_q D\theta_r+D\theta_r D\theta_q
       +\big(D\theta_q D\theta_r+D\theta_r D\theta_q\big)^T\notag\\
 &\quad+D\theta_q(D\theta_r)^T+D\theta_r(D\theta_q)^T.
 \label{eq:Aqr-pullback}
\end{align}
The first derivatives in direction $r$ follow by replacing $q$ with $r$.

\subsection{Element matrices and verified fan assembly}\label{app:assembly}

Both FreeFEM drivers construct the coarse mesh by joining the $n$ fan
triangles, then apply \texttt{trunc(coarse,1,split=m)} to subdivide each
triangle into $m^2$ triangles.  This implements the fitted mesh described
in Section~\ref{sec:residual-principle}.  Since the vertex-displacement
fields are affine on the coarse fan, their gradients and the coefficients
of the differentiated forms are constant on every fine triangle.

Let $I$ be the set of interior vertices, and let $K,M$ be the stiffness
and consistent mass matrices restricted to these vertices.  The
Faber--Krahn inequality and the local $P_1$ mass matrix give
\begin{equation}\label{eq:alpha-lower}
 K\succeq\frac{j_{0,1}^2\pi}{A}M,\qquad
 \lambda_{\min}(M)\geq
 \frac1{12}\min_{i\in I}\sum_{T\ni i}|T|,
\end{equation}
The product of the two right-hand constants is a valid
$\underline\alpha$.  Sharper verified generalized-eigenvalue bounds may be
substituted, but are not necessary for logical correctness.  The executable
uses an even more elementary bound specialized to the circumradius-one fan:
$P_n\subset(-1,1)^2$, so Dirichlet domain monotonicity gives
\begin{equation}\label{eq:square-lambda-bound}
 \lambda_1(P_n)\geq\lambda_1((-1,1)^2)=\frac{\pi^2}{2}.
\end{equation}
Together with the mass bound in \eqref{eq:alpha-lower}, this gives the
fully analytic $\underline\alpha$ used for every interval residual.

We follow the interior-submatrix pattern of
\cite{BogoselBucurValidated2024}.  Boundary coefficients are not retained as
small floating-point numbers: every Dirichlet system uses only the
interior submatrix, and its candidate is extended by exact zero on
$\partial P_n$.  On the symmetric fan all $nm^2$ triangles
are congruent and
\begin{equation}\label{eq:analytic-local-matrices}
 |T|=\frac{\sin\vartheta}{2m^2},\qquad
 M_T=\frac{|T|}{12}
 \begin{pmatrix}2&1&1\\1&2&1\\1&1&2\end{pmatrix},
 \qquad (K_T)_{ij}=|T|\nabla\lambda_i\cdot\nabla\lambda_j.
\end{equation}
For the ordering $0,m^{-1}a_0,m^{-1}a_1$, the last matrix is explicitly
\begin{equation}\label{eq:analytic-local-stiffness}
 K_T=
 \begin{pmatrix}
 \tan(\vartheta/2)&-\tan(\vartheta/2)/2&-\tan(\vartheta/2)/2\\
 -\tan(\vartheta/2)/2&(\tan(\vartheta/2)+\cot\vartheta)/2&-\cot\vartheta/2\\
 -\tan(\vartheta/2)/2&-\cot\vartheta/2&(\tan(\vartheta/2)+\cot\vartheta)/2
 \end{pmatrix};
\end{equation}
the other elements give rotated copies and permutations.  The differentiated
element matrices and loads are equally explicit:
\begin{equation}\label{eq:analytic-differentiated-matrices}
\begin{aligned}
 (K_{q,T})_{ij}&=|T|\nabla\lambda_i^T\mathbb A_q\nabla\lambda_j,
 & (K_{qr,T})_{ij}&=|T|\nabla\lambda_i^T\mathbb A_{qr}\nabla\lambda_j,\\
 f_T&=\frac{|T|}{3}(1,1,1)^T,
 & f_{q,T}&=\frac{|T|\operatorname{div}\theta_q}{3}(1,1,1)^T.
\end{aligned}
\end{equation}
The mixed load is
\[
 f_{qr,T}=\frac{|T|}{3}
 \big[(\operatorname{div}\theta_q)(\operatorname{div}\theta_r)
      -\tr(D\theta_q D\theta_r)\big](1,1,1)^T.
\]

In the implementation every trigonometric constant and every local
entry is evaluated by Arb.  Floating-point coordinates are used only to
nominate integer fan coordinates $(s,a,b)$; Arb then proves that the exported
coordinate and Fourier-direction balls contain the corresponding exact
trigonometric values.  The nomination is also proved bijective.  After this
check the decimal geometry and directions are discarded and the analytic fan
is regenerated in Arb.

The integer connectivity is checked independently.  Besides connectedness,
boundary degree two, the Euler identity, $nm$ boundary vertices, $nm^2$
triangles, and absence of duplicates, the verifier compares the full triangle
multiset with the two canonical families
\[
 (a,b),(a+1,b),(a,b+1),\qquad
 (a+1,b),(a+1,b+1),(a,b+1)
\]
in every sector.  This equality establishes the prescribed triangulated
disk and its single boundary cycle.  Hence the assembly is the analytic assembly
\eqref{eq:analytic-local-matrices}--\eqref{eq:analytic-differentiated-matrices}
on the prescribed regular fan.  FreeFEM supplies only floating-point candidate
coefficient vectors, never an authoritative matrix, geometry, or boundary
condition.

\subsection{Flux formulas}\label{app:fluxes}

The volume source in the second lifting admits an explicit equilibrated
flux because
\begin{equation}\label{eq:second-source-divergence}
 (\operatorname{div}\theta_q)(\operatorname{div}\theta_r)-\tr(D\theta_q D\theta_r)
 =\operatorname{div}\big((\operatorname{div}\theta_q)\theta_r
                         -D\theta_q\,\theta_r\big).
\end{equation}
The vector on the right has a continuous normal trace across each fitted fan
edge: continuity of a piecewise-affine field implies that the jump of its
gradient is $a\otimes n$, and the two normal-jump terms cancel.  Consequently
the negative of this explicit vector plus a curl potential is an exactly
equilibrated $H(\operatorname{div})$ flux for \eqref{eq:second-lifting}.

For the modewise computation, the four potentials
$\psi_{u,h},\psi_{q,h},\psi_{r,h},\psi_{z,h}$ are continuous and piecewise
affine and are treated as exact functions specified by their rounded nodal
values.  The flux mismatches used in \eqref{eq:combined-deltas} are
\begin{align*}
 \Phi_0&=\big\|-x/2+\operatorname{curl}\psi_{u,h}
                         -\nabla\widetilde u\big\|,\\
 \Phi_q&=\big\|-\theta_q+\operatorname{curl}\psi_{q,h}
             -\nabla\widetilde u_q-\mathbb A_q\nabla\widetilde u\big\|,\\
 \Phi_r&=\big\|-\theta_r+\operatorname{curl}\psi_{r,h}
             -\nabla\widetilde u_r-\mathbb A_r\nabla\widetilde u\big\|,\\
 \Phi_z&=\Big\|-(\operatorname{div}\theta_q)\theta_r+D\theta_q\,\theta_r
             +\operatorname{curl}\psi_{z,h}-\nabla\widetilde z\\
 &\hspace{22mm}-\mathbb A_q\nabla\widetilde u_r
             -\mathbb A_r\nabla\widetilde u_q
             -\mathbb A_{qr}\nabla\widetilde u\Big\|.
\end{align*}
The fluxes in the first-derivative equations have divergence
$-\operatorname{div}\theta_q$ and $-\operatorname{div}\theta_r$.
Equation~\eqref{eq:second-source-divergence} gives the required divergence
for the lifting flux.  The form constants in \eqref{eq:Cqr} can be taken
as the maxima, over fan triangles, of the matrix operator norms of
$\mathbb A_q,\mathbb A_r,\mathbb A_{qr}$.  Thus the mismatch norms and all
transfer constants reduce to explicit triangle calculations.

For comparison, the general residual estimate also allows a flux whose
divergence is not exactly equilibrated.

For the torsion solution, define
\begin{equation}
 R_u(v)=\int_{P_n}v-\int_{P_n}\nabla u_h\cdot\nabla v.
\end{equation}
The residual identity gives
$\|\nabla(u-u_h)\|=\|R_u\|_{H^{-1}}$, with the dual norm taken relative
to $\|\nabla v\|$.  A guaranteed functional majorant is,
for any $y_h\in H(\operatorname{div};P_n)$,
\begin{equation}\label{eq:flux-majorant}
 \|\nabla(u-u_h)\|
 \leq \|\nabla u_h-y_h\|+C_F\|1+\operatorname{div}y_h\|,
 \qquad
 C_F\leq\frac{\sqrt A}{j_{0,1}\sqrt\pi}.
\end{equation}
Indeed, integration by parts gives
\[
 R_u(v)=\int(y_h-\nabla u_h)\cdot\nabla v
       +\int(1+\operatorname{div}y_h)v;
\]
testing with the error and using the Friedrichs inequality proves
\eqref{eq:flux-majorant}.  This is a functional error majorant of the type
developed in \cite[Chapter~3]{Repin2008}; the equilibrated-flux literature
is discussed in Section~\ref{sec:residual-principle}.  An equilibrated flux
makes the second term vanish, recovering \eqref{eq:equilibrated-majorant}.

The same construction applies to the material equations.  With
$e_c$ the $c$-th coordinate vector, define
\begin{align}
 G_{i,c}(v)&=-(\partial_c\phi_i)\nabla v
 +(\partial_c v)\nabla\phi_i
 +(\nabla v\cdot\nabla\phi_i)e_c,
 &g_{i,c}&=\partial_c\phi_i.                         \label{eq:G-material}
\end{align}
Then \eqref{eq:material} is
\begin{equation}
 \int\nabla U_i^c\cdot\nabla v
 =\int G_{i,c}(u)\cdot\nabla v+\int g_{i,c}v.
\end{equation}
Choose a globally $H(\operatorname{div})$-conforming flux with the required
divergence.  Again it can be constructed without a mixed solve:
\begin{equation}\label{eq:curl-material-flux}
 q_{i,c,h}=-\phi_i e_c+\operatorname{curl}\psi_{i,c,h},
 \qquad \psi_{i,c,h}\text{ continuous and piecewise affine}.
\end{equation}
Because $\phi_i$ is globally continuous,
$\operatorname{div}q_{i,c,h}=-\partial_c\phi_i=-g_{i,c}$, including across
the fan interfaces.  With $\delta_0\geq\|u-u_h\|_a$, integration by parts gives
\begin{align}
 \|\nabla(U_i^c-U_{i,h}^c)\|
 \leq{}&\|G_{i,c}(u_h)+q_{i,c,h}-\nabla U_{i,h}^c\|
       +\|G_{i,c}(u)-G_{i,c}(u_h)\|\notag\\
 \leq{}&\|G_{i,c}(u_h)+q_{i,c,h}-\nabla U_{i,h}^c\|
       +\|\nabla\phi_i\|_{L^\infty}\delta_0.          \label{eq:material-majorant}
\end{align}
The last constant is exact: pointwise, the linear map from
$\nabla(u-u_h)$ to the difference of the two $G$ fields is a scaled orthogonal map
of norm $|\nabla\phi_i|$.  For the regular fan triangulation it equals
$1/\sin\vartheta$.  The benchmark chooses each curl potential by a scalar
least-squares solve.  The solve need not be enclosed: every rounded nodal
vector defines an admissible flux satisfying the divergence constraint
exactly; optimization affects sharpness only.

\subsection{Comparison with direct termwise bounds}\label{app:direct-bounds}

The following estimate illustrates the loss incurred by propagating state
errors separately through the distributed Hessian.  It is a diagnostic;
the sign certificates use \eqref{eq:complete-F-radius}.

Suppose the resulting bounds are
\begin{equation}
 \|\nabla(u-u_h)\|\leq\delta_0,\qquad
 \|\nabla(U_i^p-U_{i,h}^p)\|\leq\delta_{i,p}.
\end{equation}
For example, the Gram term in \eqref{eq:HJ} then satisfies
\begin{align}
|a(U_i^p,U_j^q)-a(U_{i,h}^p,U_{j,h}^q)|
\leq{}&\delta_{i,p}\|\nabla U_{j,h}^q\|
      +\delta_{j,q}\|\nabla U_{i,h}^p\|
      +\delta_{i,p}\delta_{j,q}.                    \label{eq:gram-error}
\end{align}
The quadratic $u$ terms are treated by
\begin{equation}
 \|\nabla u\otimes\nabla u-
   \nabla u_h\otimes\nabla u_h\|_{L^1}
 \leq\delta_0(2\|\nabla u_h\|+\delta_0),             \label{eq:wquad-error}
\end{equation}
multiplied by the explicit $L^\infty$ norms of the hat gradients.  The
$L^2$ state error needed for the $u$ term follows from $C_F\delta_0$.
Equations \eqref{eq:gram-error}--\eqref{eq:wquad-error}, followed by
\eqref{eq:HF-general} and the Fourier sum, give a rigorous entry radius for
each $B_k$.  Arb then evaluates \eqref{eq:mode-eigs}; an interval is certified
negative precisely when its upper endpoint is negative.

For $n=5,m=32$, at 192-bit precision and coordinate radius $10^{-13}$, Arb
gives
\begin{align}
 \delta_0&\leq0.01616456931, &
 \max_{i,c}\delta_{i,c}&\leq0.03532843568,\notag\\
 \|\nabla u_h\|&\leq0.4593971942, &
 \max_{i,c}\|\nabla U_{i,h}^c\|&\leq0.2367189574.
                                                        \label{eq:arb-majorants}
\end{align}
The displayed decimals are rounded upward from the Arb intervals.  Applying
\eqref{eq:gram-error} entry by entry then yields a rigorous
symbol-entry radius that is too large to certify the sign, namely
$0.08939022$.  This uniform propagation discards Galerkin cancellation.
At $m=64$ the validated values are $\delta_0\leq0.008112289$ and symbol
radius $0.04430895$, consistent with the first-order terms in the bound.  Proposition~\ref{prop:second-residual}
gives a quadratic residual representation.
These comparison values come from the older
\texttt{majorant\_cert.c} path with the benchmark flags
\texttt{-majorant 1} and \texttt{-export-majorant}.  It encloses the
triangle integrals and supplied candidate fields, but does not independently
establish the canonical fan incidence and boundary-node set.  The theorem
certificates use the full reconstruction in
\texttt{second\_lifting\_cert.c}.

\section{Development history and AI assistance}
\label{app:AIhistory}

The paper follows the proof strategy developed for local minimality in the
polygonal Faber--Krahn problem in
\cite{BogoselBucur2024,BogoselBucurValidated2024}:
\begin{itemize}[noitemsep]
\item Compute the Hessian with respect to vertex coordinates from the
second shape derivative, derived for torsional rigidity in
\cite{Laurain2020}.
\item At the regular polygon, change to radial--tangential coordinates to
obtain a block-circulant Hessian.  The Fourier reduction in
\cite{Tee2007}, also used in \cite{BogoselBucur2024}, expresses its
eigenvalues in terms of explicit contributions and terms involving
material derivatives.
\item Approximate the torsion state and its material derivatives by
piecewise affine finite elements.  Bound the discretization, algebraic,
and rounding errors to obtain guaranteed Hessian eigenvalue enclosures.
\item Use the four exact similarity zeros of the scale-invariant Hessian.
If its remaining $2n-4$ eigenvalues are strictly negative, the regular
$n$-gon is a strict local maximizer of torsional rigidity at fixed area.
\end{itemize}

\textbf{Models and supplied material.}
ChatGPT 5.6 Sol and 6 Astra were used through Codex in VS Code, with
reasoning settings ranging from \emph{high} to \emph{ultra}.  The
\texttt{Papers} folder contained PDF copies of
\cite{Laurain2020,BogoselBucur2024,BogoselBucurValidated2024} and the
author's notes on regular polygons and simplifications of Laurain's
polygonal shape derivative.

The author also supplied FreeFEM code for assembling the Hessian from
\cite{Laurain2020} and computing its eigenvalues.  For regular polygons,
it uses a symmetric mesh invariant under the dihedral group; the code
also handles nonregular polygons.  This initial code was developed and
tested independently by the author and served as a benchmark for the
subsequent computations.  The working environment contained installations
of Python, FreeFEM~\cite{Hecht2012}, and FLINT/Arb that had been tested in
previous projects.

The initial prompt is reproduced verbatim below:

\emph{
	Goal: local maximality of the regular n-gon for the Torsional rigidity. Follow the Bucur-Bogosel strategy: compute block circulant Hessian eigenvalues. Then attempt to prove negativity theoretically. Alternatively use validated computing to find 2n-4 negative eigenvalue in the scale invariant setting, or additive setting (use freefem codes as benchmarks). Write everything in a latex file on the theoretical side. Codes in a separate code folder. Use Flint for validated computing. Residuals for PDE-problems control error. Conjecture then attempt to find the optimal error rate in terms of mesh size for finite elements approximation. Write results in the latex file, separate sections.
}

The initial conversation is recorded in \texttt{PromptHistory.md} in the
code repository.  The main stages of development were as follows:
\begin{itemize}[noitemsep]
\item The initial draft did not contain the explicit Hessian eigenvalue
formulas corresponding to those in \cite{BogoselBucur2024}.  Further
direction from the author led to their derivation and to the addition of
the missing reference \cite{Tee2007} for the block-circulant reduction.
\item Floating-point experiments preceded the FLINT implementation.  They
examined the eigenvalue signs and distances from zero to assess the
accuracy needed for interval certification of the $2n-4$ negative
eigenvalues.
\item The error estimates were developed using the guaranteed-flux
principle discussed in \cite{ErnVohralik2015} and the explicit flux-plus-curl
construction of Section~\ref{sec:residual-principle}.  The implementation
does not use the mixed patch problems of \cite{ErnVohralik2015}.  The
second-variation Galerkin identity in Proposition~\ref{prop:second-residual}
retains products of energy errors in the Hessian bound.  The earlier
certificates in \cite{BogoselBucurValidated2024} establish the cases
$n=5,6$ for the first Dirichlet eigenvalue.  Since that objective differs
from torsion, the certified ranges alone do not measure the improvement
in error estimation.
\item For the pentagon, an initial prototype produced four intervals
containing zero and six strictly negative intervals.  Containment of zero
did not prove the four exact zeros, and the prototype did not yet provide
a complete continuous-error certificate.
\item Attempts to prove negativity analytically by exploring further
relations between the eigenvalues were unsuccessful.
\item A subsequent review identified the missing steps in the pentagon
certificate.  The author supplied the repository accompanying
\cite{BogoselBucurValidated2024} as a model for controlling the errors of
floating-point finite-element solutions through their residuals.  The
resulting complete pentagon certificate was then extended to $n=10$.
\item For the extension to $n=25$, the author requested that the code use
dihedral symmetry to recover all first material derivatives from the two
solutions associated with the first vertex.  The previous implementation
computed them separately for all $2n$ vertex coordinates.
\item An adversarial audit of the theoretical and computational arguments
identified inaccuracies that were subsequently corrected.
\end{itemize}

The author then reviewed the manuscript and made or requested the
following revisions:
\begin{itemize}[noitemsep]
\item Remove unnecessary notation, including abbreviations for
$\vartheta/2$ and $k \vartheta$, and symbols introduced for a single use.
\item Remove theoretical developments belonging to incomplete proof
attempts.
\item Add references on recent advances in torsion, validated PDE
computations, and spectral geometry.
\item Reorganize the error analysis around discretization errors,
algebraic errors, and rounding errors controlled by interval arithmetic.
\end{itemize}

The development history and code are available in the GitHub repository:
\begin{center}
\href{https://github.com/beniamin-bogosel/MaxTorsionValidation}{\nolinkurl{https://github.com/beniamin-bogosel/MaxTorsionValidation}}
\end{center}
The repository includes the original FreeFEM benchmarks and the interval
verification code.  The \texttt{DetailedExplanations} folder contains
three lectures covering the Hessian eigenvalue formulas, the error
estimates for certification, and the FLINT/Arb implementation.


\textbf{Data availability.} The source code and instructions for
regenerating the numerical certificates are available at the following
repository.  Section~\ref{sec:reproducibility} gives a small pentagon example
and the commands for the reported tables.  The large generated candidate
archives are excluded from the source distribution.  Readers can produce
new candidate fields, residual bounds, and complete sign certificates
from the supplied code; replaying historical candidates is a separate
option when those files are available.
\begin{center}
	\href{https://github.com/beniamin-bogosel/MaxTorsionValidation}{\nolinkurl{https://github.com/beniamin-bogosel/MaxTorsionValidation}}
\end{center}

\textbf{Conflicts of interest.} The author declares no conflicts of interest.

\textbf{AI usage disclosure.} Appendix~\ref{app:AIhistory} describes the use
of AI in this project.  The author reviewed the theoretical results and
numerical code and takes full responsibility for the contents of the paper.

\bibliographystyle{abbrv}
\bibliography{torsional_rigidity_regular_polygon}

\end{document}